\documentclass{amsart} 

\usepackage{amsmath,amsthm}     
\usepackage{amssymb,empheq}            
\usepackage{euscript}           
\usepackage{graphicx,enumerate,calc,lscape,color}
\definecolor{red}{rgb}{1,0,0}
\usepackage[matrix,arrow,curve,frame]{xy}    
\usepackage{hyperref}

\xymatrixcolsep{1.9pc}                          
\xymatrixrowsep{1.9pc}
\newdir{ >}{{}*!/-5pt/\dir{>}}                  

\newtheorem{thm}[subsection]{Theorem}
\newtheorem{defn}[subsection]{Definition}
\newtheorem{prop}[subsection]{Proposition}

\newtheorem{cor}[subsection]{Corollary}
\newtheorem{lemma}[subsection]{Lemma}

\theoremstyle{definition}  
\newtheorem{ex}[subsection]{Example}

\newtheorem{remark}[subsection]{Remark}

  {\end{list}}

\newcommand{\dfn}{\textbf} 

\newcommand{\Wedge}             {\vee}

\newcommand{\iso}               {\cong}

\newcommand{\cat}{\EuScript}    
\newcommand{\cC}{{\cat C}}
\newcommand{\cD}{{\cat D}}

\newcommand{\cP}{{\cat P}}

\newcommand{\Top}{{\cat Top}}
\newcommand{\kG}{{\mathcal G}}

\newcommand{\Set}{{\cat Set}}

\newcommand{\sSet}{{s{\cat Set}}}

\newcommand{\Grph}{{\cat Grph}}

\newcommand{\Cat}{{\cat Cat}}

\newcommand{\field}[1]  {\mathbb #1} 
\newcommand{\A}         {\field A}

\newcommand{\N}         {\field N}

\DeclareMathOperator*{\colim}{colim}

\DeclareMathOperator*{\hocolim}{hocolim}

\DeclareMathOperator{\Hom}{Hom}

\DeclareMathOperator{\ob}{ob}

\DeclareMathOperator{\coeq}{coeq}

\DeclareMathOperator{\id}{id}

\newcommand{\ra}{\rightarrow}                   
\newcommand{\lra}{\longrightarrow}              
\newcommand{\la}{\leftarrow}                    
\newcommand{\llra}[1]{\stackrel{#1}{\lra}}      

\newcommand{\fib}{\twoheadrightarrow}           

\newcommand{\inc}{\hookrightarrow}              
\newcommand{\dbra}{\rightrightarrows}           

\newcommand{\blank}{-}                          

\newcommand{\ovcat}{\downarrow}

\newcommand{\bd}[1]{\partial\Delta^{#1}}

\newcommand{\adjoint}{\rightleftarrows}

\newcommand{\bdd}[1]{\partial\Delta^{#1}}

\newcommand{\del}[1]{\Delta^{#1}}

\newcommand{\he}{\simeq}

\newcommand{\rea}[1]{|{#1}|}             

\newcommand{\ceck}[1]{\Cech(#1)}         
\newcommand{\oceck}[1]{\Cech^{o}(#1)}    
\newcommand{\oreal}[1]{\rea{\oceck{U}}}  
\newcommand{\creal}[1]{\rea{\ceck{U}}}   

\newcommand{\Cech}{\check{C}}

\newcommand{\V}[2]{V_{#1}(\A^{#2})}

\numberwithin{equation}{subsection}

\newenvironment{myequation}
  {\addtocounter{subsection}{1}\begin{eqnarray}}
  {\end{eqnarray}$\!\!$}

\newcommand{\jC}{\mathfrak{C}}

\newcommand{\jCn}{\mathfrak{C}^{nec}}
\newcommand{\jCh}{\mathfrak{C}^{hoc}}

\newcommand{\Nec}{{\cat{N}ec}}

\DeclareMathOperator{\Spi}{Spi}

\newcommand{\nosee}[1]{}

\def\tn{\textnormal}

\def\id{\tn{id}}
\def\to{\ra}
\def\To{\xrightarrow}
\def\too{\longrightarrow}
\def\fromm{\longleftarrow}
\def\from{\leftarrow}
\def\From{\xleftarrow}
\def\vect{\overrightarrow}
\def\taking{\colon}
\def\sCat{{s\Cat}}
\def\poleq{\preceq}
\def\inj{\hookrightarrow}
\def\cross{\times}
\def\wt{\widetilde}

\newcommand{\Adjoint}[4]{\xymatrix@1{#2 \ar@<.5ex>[r]^-{#1} & #3 \ar@<.5ex>[l]^-{#4}}}
\newcommand{\bpSet}{\sSet_{*,*}}
\def\ss{\subseteq}

\begin{document}

\title{Rigidification of quasi-categories}

\author{Daniel Dugger}
\author{David Spivak}

\address{Department of Mathematics\\ University of Oregon\\ Eugene, OR
97403}

\address{Department of Mathematics\\ University of Oregon\\ Eugene, OR
97403}

\email{ddugger@uoregon.edu}

\email{dspivak@uoregon.edu}

\begin{abstract}
We give a new construction for rigidifying a quasi-category into a
simplicial category, and prove that it is weakly equivalent to the
rigidification given by Lurie.  Our construction comes from the use of
necklaces, which are simplicial sets obtained by stringing simplices
together.  As an application of these methods, we use our model to
reprove some basic facts from \cite{L} about the rigidification
process.
\end{abstract}

\maketitle

\tableofcontents

\section{Introduction}
Quasi-categories are a certain generalization of categories, in which
one has not only $1$-morphisms but $n$-morphisms for every natural
number $n$.  They have been extensively studied by Cordier and Porter \cite{CP}, by Joyal \cite{J1}, 
\cite{J2}, and by Lurie \cite{L}.  If $K$ is a quasi-category and $x$
and $y$ are two objects of $K$, then one may associate a ``mapping
space'' $K(x,y)$ which is a simplicial set.  There are many different
constructions for these mapping spaces, but in \cite{L} one particular
model is given for which there are composition maps
$K(y,z)\times K(x,y)\ra K(x,z)$ giving rise to a simplicial category.
This simplicial category is denoted $\jC(K)$, and it may be thought of
as a {\it rigidification\/} of the quasi-category $K$.  It is proven
in \cite{L} that the homotopy theories of quasi-categories and
simplicial categories are equivalent via this functor.

In this paper we introduce some new models for the mapping spaces
$K(x,y)$, which are particularly easy to describe and particularly
easy to use---in fact they are just the nerves of ordinary categories
(i.e., 1-categories).  Like Lurie's model, our models admit
composition maps giving rise to a simplicial category; so we are
giving a new method for rigidifying quasi-categories.  We prove that
our construction is homotopy equivalent (as a simplicial category) to
Lurie's $\jC(K)$.  Moreover, because our mapping spaces are
nerves of categories there are many standard tools available for
analyzing their homotopy types.  We demonstrate the effectiveness of
this by giving new proofs of some basic facts about the functor
$\jC(\blank)$.

One payoff of this approach is that it is possible to give a streamlined proof
of Lurie's Quillen equivalence between the homotopy theory of
quasi-categories and simplicial categories.  This requires, however, a
more detailed study of the model category structure on
quasi-categories.  We will take this up in a sequel \cite{DS} and prove the
Quillen equivalence there.

\medskip

\subsection{Mapping spaces via simplicial categories}\label{subsec:mapping in sCat}
Now we describe our results in more detail.  A quasi-category
is a simplicial set that has the right-lifting-property 
with respect to inner horn inclusions $\Lambda^n_i\to\Delta^n,\;\;\;
0<i<n$.  It turns out that there is a unique model structure on
$\sSet$ where the cofibrations are the monomorphisms and the fibrant
objects are the quasi-categories; this will be called the {\it Joyal
model structure\/} and denoted $\sSet_J$.  The weak equivalences in
$\sSet_J$ will here be called {\it Joyal equivalences\/}.  The
existence of the Joyal model structure will not be needed in this
paper, although it provides some useful context.  The notions of
quasi-categories and Joyal equivalences, however, will be used in
several places.  See 
Section~\ref{subsec:quasi-categories} for additional background.

There is a functor, constructed in \cite{L}, which sends any
simplicial set $K$ to a corresponding simplicial category $\jC(K)\in\sCat$.  
This is the left adjoint in a Quillen pair
\[ \jC \colon \sSet_J \adjoint s\Cat\colon N, \]
where $N$ is called the \dfn{coherent nerve}.  The functor $N$ can be
described quite explicitly (see Section~\ref{se:background}), but the
functor $\jC$ is in comparison a little mysterious.  In \cite{L} each
$\jC(K)$ is defined as a certain colimit in the category $s\Cat$, but
colimits in $\sCat$ are notoriously difficult to understand.

Our main goal in this paper is to give a different model for the
functor $\jC$. Define a \dfn{necklace} (which we picture as
``unfastened") to be a simplicial set of the form
\[ \Delta^{n_0}\Wedge\Delta^{n_1}\Wedge \cdots \Wedge\Delta^{n_k}\]
where each $n_i\geq 0$ and where in each wedge the final vertex of
$\Delta^{n_i}$ has been glued to the initial vertex of
$\Delta^{n_{i+1}}$. 

The first and last vertex in any necklace $T$ are denoted $\alpha_T$
and $\omega_T$, respectively (or just $\alpha$ and $\omega$ if $T$ is
obvious from context).  If $S$ and $T$ are two necklaces, then by
$S\Wedge T$ we mean the necklace obtained in the evident way, by
gluing the final vertex $\omega_S$ of $S$ to the initial vertex
$\alpha_T$ of $T$.  Write $\Nec$ for the category whose objects are
necklaces and where a morphism is a map of simplicial sets which preserves the
initial and final vertices.  

Let $S\in \sSet$ and let $a,b\in S_0$.  If $T$ is a necklace, we use
the notation $$T\to S_{a,b}$$ to indicate a morphism of simplicial
sets $T\ra S$ which sends $\alpha_T$ to $a$ and $\omega_T$ to $b$.  Let
$(\Nec\ovcat S)_{a,b}$ denote the evident category whose objects are
pairs $[T,T\ra S_{a,b}]$ where $T$ is a necklace.  Note that for
$a,b,c\in S$, there is a functor
\[ (\Nec\ovcat S)_{b,c}\times (\Nec\ovcat S)_{a,b} \lra (\Nec\ovcat
S)_{a,c}
\]
which sends the pair $[T_2,T_2\ra S_{b,c}] \times [T_1,T_1\ra S_{a,b}]$ to
$[T_1\Wedge T_2,T_1\Wedge T_2\ra S_{a,c}]$.

Let $\jCn(S)$ be the function which assigns to any $a,b\in S_0$ the
simplicial set $\jCn(S)(a,b)=N(\Nec\ovcat S)_{a,b}$ (the classical nerve
of the $1$-category $(\Nec\ovcat S)_{a,b}$).  The above pairings of
categories induces pairings on the nerves, which makes $\jCn(S)$ into
a simplicial category with object set $S_0$.  

\begin{thm}
\label{th:main1}
There is a natural zig-zag of weak equivalences of simplicial
categories between $\jCn(S)$
and $\jC(S)$, for all simplicial sets $S$.
\end{thm}

In the above result, the weak equivalences for simplicial categories
are the so-called ``DK-equivalences'' used by Bergner in
\cite{Bergner}.  See Section~\ref{se:background} for this notion.

In this paper we also give an explicit description of the mapping
spaces in the simplicial
category $\jC(S)$.  A rough statement is given below, but see
Section~\ref{sec:jC} for more details.

\begin{thm}
Let $S$ be a simplicial set and let $a,b\in S$.  Then the mapping
space $X=\jC(S)(a,b)$ is the simplicial set whose $n$-simplices
are triples subject to a certain equivalence relations.  The triples
consist of a necklace $T$, a map $T\ra S_{a,b}$, and a flag
$\vect{T}=\{ T^0\subseteq\cdots\subseteq T^n \}$ of vertices in $T$.  
For the equivalence relation, see Corollary~\ref{co:description1}.
The face maps and
degeneracy maps are obtained by removing or repeating elements $T^i$
in the flag.

\par\noindent
The pairing
\[ \jC(S)(b,c) \times \jC(S)(a,b) \lra \jC(S)(a,c) 
\]
sends the pair of $n$-simplices $([T\ra S; \vect{T^i}],[U\ra S,\vect{U^i}])$ to 
$[U\Wedge T\ra S,\vect{U^i\cup T^i}]$.
\end{thm}

Theorem~\ref{th:main1} turns out to be very useful in the study of the
functor $\jC$.  There are many tools in classical homotopy theory for
understanding the homotopy types of nerves of $1$-categories, and via
Theorem~\ref{th:main1} these tools can be applied to understand mapping
spaces in $\jC(S)$.  We demonstrate this technique in
Section~\ref{se:properties} by proving, in a new way, the following
two properties of $\jC$ found in \cite{L}.

\begin{thm} Let $X$ and $Y$ be simplicial sets.
\begin{enumerate}[(a)]
\item The natural map
$\jC(X\times Y) \ra \jC(X)\times \jC(Y)$ is a weak equivalence of
simplicial categories;
\item If $X\ra Y$ is a Joyal equivalence  then $\jC(X)\ra
\jC(Y)$ is a weak equivalence.
\end{enumerate}
\end{thm}

%

\subsection{Notation and Terminology}
We will sometimes use $\sSet_K$ to refer to the usual model structure
on simplicial sets, which we'll term the {\it Kan model structure\/}.
The fibrations are the Kan fibrations, the weak equivalences (called
Kan equivalences from now on) are the maps which induce homotopy
equivalences on geometric realizations, and the cofibrations are the
monomorphisms.

We will often be working with the category $\bpSet=(\bdd{1}\ovcat\sSet)$.
Note that $\Nec$ is a
full subcategory of $\bpSet$.

An object of $\bpSet$ is a simplicial set $X$ with two distinguished
points $a$ and $b$.  We sometimes (but not always) write $X_{a,b}$ for $X$, to remind
us that things are taking place in $\bpSet$ instead of $\sSet$.  

If $\cC$ is a (simplicial) category containing objects $X$ and $Y$, we
write $\cC(X,Y)$ for the (simplicial) set of morphisms from $X$ to $Y$.
\section{Background on quasi-categories}
\label{se:background}

In this section we give the background on quasi-categories and
simplicial categories needed in the rest of the paper.

\subsection{Simplicial categories}
A simplicial category is a category enriched over simplicial sets; it
can also be thought of as a simplicial object of $\Cat$ in which the
categories in each level have the same object set.  We use $\sCat$ to
denote the category of simplicial categories.  A cofibrantly-generated
model structure on
$\sCat$ was developed in \cite{Bergner}.  A map of simplicial
categories $F\colon\cC \ra \cD$ is a weak equivalence (sometimes called a
{\it DK-equivalence\/}) if
\begin{enumerate}[(1)]
\item For all $a,b\in \ob\cC$, the map $\cC(a,b) \ra \cD(Fa,Fb)$ is a
Kan equivalence of simplicial sets; 
\item The induced functor of ordinary categories $\pi_0 F\colon \pi_0
\cC \ra \pi_0\cD$ is surjective on isomorphism classes.  
\end{enumerate}  
Likewise, the map $F$ is a fibration if
\begin{enumerate}[(1)]
\item For all $a,b\in \ob\cC$, the map $\cC(a,b) \ra \cD(Fa,Fb)$ is a
Kan fibration of simplicial sets;
\item For all $a \in \ob\cC$ and $b\in \ob\cD$, if $e\colon Fa \ra b$
is a map in $\cD$ which becomes an isomorphism in $\pi_0\cD$, then
there is an object $b'\in \cC$ and a map $e'\colon a\ra b'$ such that
$F(e')=e$ and $e'$ becomes an isomorphism in $\pi_0\cC$.
\end{enumerate}  
The cofibrations are the maps which have the
left lifting property with respect to the acyclic fibrations.

\begin{remark}
The second part of the fibration condition seems a little awkward at
first.  In this paper we will actually have no need to think about fibrations
of simplicial categories, but have included the definition for
completeness.  

Bergner writes down sets of generating cofibrations and acyclic
cofibrations in \cite{Bergner}. 
\end{remark}

\subsection{Quasi-categories and Joyal equivalences}\label{subsec:quasi-categories}
As mentioned in the introduction, 
there is a unique model structure on
$\sSet$ with the properties that
\begin{enumerate}[(i)]
\item The cofibrations are the monomorphisms;
\item The fibrant objects are the quasi-categories.
\end{enumerate}

It is easy to see that there is at most one such structure.  To do
this, let $E^1$ be the $0$-coskeleton---see \cite{AM}, for
instance---of the set $\{0,1\}$ (note that the geometric realization
of $E^1$ is essentially the standard model for $S^\infty$).  The map
$E^1\ra *$ has the right lifting property with respect to all
monomorphisms, and so it will be an acyclic fibration in this
structure.  Therefore $X\times E^1\ra X$ is also an acyclic fibration
for any $X$, and hence $X\times E^1$ will be a cylinder object for
$X$.  Since every object is cofibrant, a map $A\ra B$ will be a weak
equivalence if and only if it induces bijections $[B,Z]_{E^1} \ra
[A,Z]_{E^1}$ for every quasi-category $Z$, where $[A,Z]_{E^1}$ means
the coequalizer of $\sSet(A\times E^1,Z)\dbra \sSet(A,Z)$.  Therefore
the weak equivalences are determined by properties (i)--(ii), and
since the cofibrations and weak equivalences are determined so are the
fibrations.

Motivated by the above discussion, we define a map of simplicial sets $A\ra B$ to be a
\dfn{Joyal equivalence} if it induces bijections
$[B,Z]_{E^1}\to[A,Z]_{E^1}$ for every quasi-category $Z$.

That there actually exists a model stucture satisfying (i) and (ii) is 
not so clear, but it was estalished by Joyal (see \cite{J1} or \cite{J2}, or
\cite{L} for another proof).  For this reason, we will call it the
\dfn{Joyal model structure} and denote it by $\sSet_J$.  The weak
equivalences are defined differently in both \cite{J2} and \cite{L},
but of course turn out to be equivalent to the definition we have
adopted here.  

In the rest of the paper we will never use the Joyal model structure,
only the notion of Joyal equivalence.

\subsection{Background on $\jC$ and $N$}

Given a simplicial category $S$, one can
construct a simplicial set called the {\it coherent
nerve\/} of $S$ \cite[1.1.5]{L}.  We will now describe this construction.

Recall the adjoint functors $F\colon \Grph \adjoint \Cat\colon U$.
Here $\Cat$ is the category of $1$-categories, and $\Grph$ is the
category of graphs: a graph consists of an object set and morphism
sets, but no composition law.  The functor $U$ is a forgetful functor,
and $F$ is a free functor.  Given any category $\cC$ we may then
consider the comonad resolution $(FU)_\bullet(\cC)$ given by
$[n]\mapsto (FU)^{n+1}(\cC)$.  This is a simplicial category.

There is a functor of simplicial categories $(FU)_\bullet(\cC) \ra
\cC$ (where the latter is considered a discrete simplicial category).
This functor induces a weak equivalence on all mapping spaces, a fact
which can be seen by applying $U$, at which point the comonad
resolution picks up a contracting homotopy.  Note that this means that
the simplicial mapping spaces in $(FU)_\bullet(\cC)$ are all homotopy
discrete.

Recall that $[n]$ denotes the category $0\ra 1 \ra \cdots \ra n$, where
there is a unique map from $i$ to $j$ whenever $i\leq j$.  
We let $\jC(\Delta^n)$ denote the simplicial category
$(FU)_{\bullet}([n])$.  The mapping spaces in this simplicial category
can be analyzed completely, and are as follows.  For each $i$ and $j$,
let $P_{i,j}$ denote the poset of all subsets of $\{i,i+1,\ldots,j\}$
containing $i$ and $j$ (ordered by inclusion).  Note that the nerve of
$P_{i,j}$ is isomorphic to the cube $(\Delta^1)^{j-i-1}$ if $j>i$,
$\Delta^0$ if $j=i$, and the emptyset if $j<i$.  The nerves of the $P_{i,j}$'s
naturally form the mapping spaces of a simplicial category with object
set $\{0,1,\ldots,n\}$, using the pairings $P_{j,k}\times P_{i,j} \ra
P_{i,k}$ given by union of sets.  

\begin{lemma}
\label{le:C(Delta^n)}
There is an isomorphism of simplicial categories $\jC(\Delta^n)\iso
NP$.  
\end{lemma}

\begin{remark}
The proof of the above lemma is a bit of an aside from the main thrust
of the paper, so it is given in Appendix~\ref{se:appendix}.  In fact
we could have {\it defined\/} $\jC(\Delta^n)$ to be $NP$, which is
what Lurie does in \cite{L}, and avoided the lemma entirely; the
construction $(FU)_\bullet([n])$ will never again be used in this
paper.  Nevertheless, the identification of $NP$ with
$(FU)_{\bullet}([n])$ seems informative to us.
\end{remark}

For any simplicial category $\cD$, the \dfn{coherent nerve} of $\cD$
is the simplicial set $N\cD$ given by
\[ [n]\mapsto s\Cat(\jC(\Delta^n),\cD).\]
It was proven by Lurie \cite{L} that $N\cD$ is
always a quasi-category; see also Lemma~\ref{lemma:N preserves fibrants} below.

The functor $N$ has a left adjoint denoted $\jC\colon \sSet \ra
s\Cat$.  Any simplicial set $K$ may be written as a colimit of
simplices via the formula
\[ K\iso \colim_{\Delta^n\ra K} \Delta^n, 
\]
and consequently one has
\begin{align}\label{dia:jC colim} \jC(K)\iso \colim_{\Delta^n\ra K} \jC(\Delta^n)
\end{align}
where the colimit takes place in $s\Cat$.
This formula is a bit unwieldy, however, in the sense that it
does not give much concrete information about the mapping
spaces in $\jC(K)$.  The
point of the next three sections is to obtain such concrete
information, via the use of necklaces.

\section{Necklaces}
\label{se:necklaces}

A necklace is a simplicial set obtained by stringing simplices
together in succession.  In this section we establish some basic facts
about them, as well as facts about the more general category of
ordered simplicial sets.  When $T$ is a necklace we are able to give a
complete description of the mapping spaces in $\jC(T)$ as nerves of
certain posets, generalizing what was said for $\jC(\Delta^n)$ in the
last section.  See Proposition~\ref{prop:mapping in necklaces}.  

\medskip

As briefly discussed in the introduction, a \dfn{necklace} is defined to be a simplicial set of the form
\[ \Delta^{n_0}\Wedge\Delta^{n_1}\Wedge \cdots \Wedge\Delta^{n_k}\]
where each $n_i\geq 0$ and where in each wedge the final vertex of
$\Delta^{n_i}$ has been glued to the initial vertex of
$\Delta^{n_{i+1}}$.  We say that the necklace is in \dfn{preferred
form} if either $k=0$ or each $n_i\geq 1$.

Let $T=\Delta^{n_0}\Wedge\Delta^{n_1}\Wedge \cdots \Wedge\Delta^{n_k}$
be in preferred form.  Each $\Delta^{n_i}$ is called a \dfn{bead} of
the necklace.  A \dfn{joint} of the necklace is either an initial or a
final vertex in some bead.  Thus, every necklace has at least one
vertex, one bead, and one joint; $\Delta^0$ is not a bead in any
necklace except in the necklace $\Delta^0$ itself.

Given a necklace $T$, write $V_T$ and $J_T$ for the sets of vertices
and joints of $T$.  Note that $V_T=T_0$ and $J_T\subseteq V_T$.  Both $V_T$ and
$J_T$ are totally ordered, by saying $a\leq b$ if there is a directed
path in $T$ from $a$ to $b$.  The initial and final vertices of $T$ are denoted $\alpha_T$ and $\omega_T$ (and we sometimes drop the subscript); note that $\alpha_T,\omega_T\in J_T$.

Every necklace $T$ comes with a particular map $\bdd{1}\ra T$ which
sends $0$ to the initial vertex of the necklace, and $1$ to the final
vertex.  If $S$ and $T$ are two necklaces, then by $S\Wedge T$ we mean
the necklace obtained in the evident way, by gluing the final vertex
of $S$ to the initial vertex of $T$.  Let $\Nec$ denote the full
subcategory of $\bpSet=(\bdd{1}\ovcat\sSet)$ whose objects are
necklaces $\bdd{1}\to T$.  We sometimes talk about $\Nec$ as though it
is a subcategory of $\sSet$.

A simplex is a necklace with one bead.  A \dfn{spine} is a necklace in
which every bead is a $\Delta^1$.  Every necklace $T$ has an
associated simplex and spine, which we now define.  Let $\Delta[T]$ be
the simplex whose vertex set is the same as the (ordered) vertex set
of $T$.  Likewise, let $\Spi[T]$ be the longest spine inside of $T$.
Note that there are inclusions $\Spi[T]\inc T\inc \Delta[T]$.  The
assignment $T\ra \Delta[T]$ is a functor, but $T\ra \Spi[T]$ is not
(for instance, the unique map of necklaces $\Delta^1\to\Delta^2$ does
not induce a map on spines).

\subsection{Ordered simplicial sets}
If $T\ra T'$ is a map of necklaces, then the image of $T$ is also a
necklace.  To prove this, as well as for several other reasons
scattered thoughout the paper, it turns out to be very convenient to
work in somewhat greater generality.  

If $X$ is a simplicial set, define a relation on its $0$-simplices by
saying that $x\poleq y$ if there exists a spine $T$ and a map $T\ra
X$ sending $\alpha_T\mapsto x$ and $\omega_T\mapsto y$.  In other words,
$x\poleq y$ if there is a directed path from $x$ to $y$ inside of
$X$.  Note that this relation is clearly reflexive and transitive, but
not necessarily antisymmetric: that is, if $x\poleq y$ and $y\poleq x$
it need not be true that $x=y$.  

\begin{defn}
\label{de:ordered}
A simplicial set $X$ is \dfn{ordered} if
\begin{enumerate}[(i)]
\item The relation $\poleq$ defined on $X_0$ is antisymmetric, and
\item An simplex $x\in X_n$ is determined by its
sequence of vertices $x(0)\poleq\cdots\poleq x(n)$; i.e. no two
distinct $n$-simplices have identical vertex sequences.
\end{enumerate} 

\end{defn}

Note the role of degenerate simplices in condition (ii).  For example,
notice that $\Delta^1/\bdd{1}$ is not an ordered simplicial set.

The following notion is also useful:

\begin{defn}
Let $A$ and $X$ be simplicial sets.  A map $A\ra X$ is called
a \dfn{simple inclusion} if it has the right lifting property with
respect to the canonical inclusions $\bd{1}\inc T$ for all necklaces
$T$.  (Note that such a map really is an inclusion, because it has the
lifting property for $\bd{1}\ra \Delta^0$).
\end{defn}

The notion of simple inclusion says that if there is a
``path'' (in the sense of a necklace) in $X$ that starts and ends in
$A$, then it must lie entirely in $A$.  As an example, four out of the five inclusions $\Delta^1\inj\Delta^1\cross\Delta^1$ are simple inclusions.  

\begin{lemma}
\label{le:simple}
A simple inclusion $A\inc X$ has the right lifting property with
respect to the maps $\bd{k}\inc \Delta^k$ for all $k\geq 1$.  
\end{lemma}

\begin{proof}
Suppose given a square
\[ \xymatrix{ \bd{k}\ar[r]\ar[d] & A \ar[d] \\
 \Delta^k \ar[r] & X.
}
\]
By restricting the map $\bd{k}\ra A$ to $\bd{1}\inc \bd{k}$ (given by
the initial and final vertices of $\bd{k}$), we get a corresponding
lifting square with $\bd{1}\inc \del{k}$.  Since $A\ra X$ is a simple
inclusion, this new square has a lift $l\colon \Delta^k \ra A$.  It is
not immediately clear that $l$ restricted to $\bd{k}$ equals our
original map, but the two maps are equal after composing with $A\ra X$;
since $A\ra X$ is a monomorphism, the two maps are themselves equal.    
\end{proof}

\begin{lemma}\label{lemma:facts on ordered}
Let $X$ and $Y$ denote ordered simplicial sets
and let $f\colon X\ra Y$ be a map.
\begin{enumerate}
\item The category of ordered simplicial sets is closed under taking
finite limits.
\item Every necklace is an
ordered simplicial set.
\item If $X'\subseteq X$ is a simplicial subset, then $X'$ is also
ordered.
\item The map $f$ is completely determined by the map $f_0\taking X_0\to
Y_0$ on vertices.
\item If $f_0$ is injective then so is $f$.
\item The image of an $n$-simplex $x\taking\Delta^n\to X$ is of the
form $\Delta^k\inj X$ for some $k\leq n$.
\item If $T$ is a necklace and $y\colon T\to X$ is a map, then its
image is a necklace.
\item Suppose that 
$X\from A\to Y$ is a diagram of ordered simplicial sets, and both $A\to X$
and $A\to Y$ are simple inclusions.   Then the pushout
$B=X\amalg_AY$ is an ordered simplicial set, and the inclusions $X\inc B$ and $Y\inc
B$ are both simple.

\end{enumerate}

\end{lemma}

\begin{proof}
For (1), the terminal object is a point with its unique
ordering.  Given a diagram of the form
$$X\too Z\fromm Y,$$ 
let $A=X\cross_Z Y$.  It is clear  that if $(x,y)\poleq_A(x',y')$ then both $x\poleq_Xx'$
and $y\poleq_Yy'$ hold, and so antisymmetry of $\poleq_A$ follows from
that of $\poleq_X$ and $\poleq_Y$.  Condition (ii) from
Definition~\ref{de:ordered} is easy to check.

Parts (2)--(5) are easy, and left to the reader.

For (6), the sequence $x(0),\ldots,x(n)\in X_0$ may have
duplicates; let $d\taking \Delta^k\to\Delta^n$ denote any face such
that $x\circ d$ contains all vertices $x(j)$ and has no duplicates.
Note that $x\circ d$ is an injection by (5).  A certain degeneracy of $x\circ
d$ has the same vertex sequence as $x$.  Since $X$ is ordered, $x$ is this
degeneracy of $x\circ d$.  Hence, $x\circ d\taking\Delta^k\inj X$ is
the image of $x$.

Claim (7) follows from (6).  

For claim (8) we first show that the maps $X\inc B$ and $Y\inc B$ are
simple inclusions.  To see this, suppose that $u,v\in X$ are vertices,
$T$ is a necklace, and $f\taking T\ra B_{u,v}$ is a map; we want to
show that $f$ factors through $X$.  Note that any simplex $\Delta^k\to
B$ either factors through $X$ or through $Y$.  Suppose that $f$ does
not factor through $X$.  From the set of
beads of $T$ which do not factor through $X$, take 
any maximal subset $T'$ in which all the beads are
adjacent.  Then we have a necklace $T'\ss T$ such that $f(T')\ss
Y$. 
If there exists a bead in $T$ prior to $\alpha_{T'}$, then it must map
into $X$ since $T'$ was maximal; so 
$f(\alpha_{T'})$ lies in  $X\cap Y=A$.  Likewise, if
there is no bead prior to $\alpha_{T'}$ then $f(\alpha_{T'})=u$ and so
again $f(\alpha_{T'})$ lies in $X\cap Y=A$.  Similar remarks apply to
show that $f(\omega_{T'})$ lies in $A$.  At this point the fact that
$A\inc Y$ is a simple inclusion implies that 
$f(T')\ss A\ss X$, which is a contradiction.  So in fact $f$ factored
through $X$.

We have shown that $X\inc B$ (and dually $Y\inc B$) is a simple
inclusion.  Now we show that $B$ is ordered, so suppose $u,v\in B$ are
such that $u\poleq v$ and $v\poleq u$.  There there are spines $T$ and
$U$ and maps $T\ra B_{u,v}$, $U\ra B_{v,u}$.  Consider the composite
spine $T\Wedge U \ra B_{u,u}$.  If $u\in X$, then since $X\inc B$ is a
simple inclusion it follows that the image of $T\Wedge U$ maps
entirely into $X$; so $u\poleq_X v$ and $v\poleq_X u$, which means
$u=v$ because $X$ is ordered.  The same argument works if $u\in Y$, so
this verifies antisymmetry of $\poleq_B$.

To verify condition (ii) of Definition~\ref{de:ordered}, suppose
$p,q\taking\Delta^k\to B$ are $k$-simplices with the same sequence of
vertices; we wish to show $p=q$.
We know that $p$ factors through $X$ or $Y$, and so does $q$; if both
factor through $Y$, then the fact that $Y$ is ordered implies that
$p=q$ (similarly for $X$).  So we may assume $p$ factors through $X$
and $q$ factors through $Y$.  By induction on $k$, the restrictions
$p|_{\bdd{k}}=q|_{\bdd{k}}$ are equal, hence factor through $A$.  By
Lemma~\ref{le:simple} applied to $A\inc X$, the map $p$ factors
through $A$.  Therefore it also factors through $Y$, and now we are
done because $q$ also factors through $Y$ and $Y$ is ordered.
\end{proof}

\subsection{Categorification of necklaces}

Let $T$ be a necklace.  Our next goal is to give a complete
description of the simplicial category $\jC(T)$.  The object set of
this category is precisely $T_0$. 

For vertices $a,b\in T_0$, let $V_T(a,b)$ denote the set of vertices
in $T$ between $a$ and $b$, inclusive (with respect to the relation
$\poleq$).  Let $J_T(a,b)$ denote the union of $\{a,b\}$ with the set
of joints between $a$ and $b$.  There is a unique subnecklace of $T$ with joints
$J_T(a,b)$ and vertices $V_T(a,b)$; let
$\wt{B}_0,\wt{B}_1,\ldots\wt{B}_k$ denote its beads.  There are
canonical inclusions of each $\wt{B}_i$ to $T$.  Hence, there is a
natural map
\[
\jC(\wt{B}_{k})(j_k,b)
\times \jC(\wt{B}_{k-1})(j_{k-1},j_k) \times \cdots \times 
\jC(\wt{B}_1)(j_1,j_2) \times \jC(\wt{B}_0)(a,j_1)
\ra \jC(T)(a,b)
\]
obtained by first including the $\wt{B}_i$'s into $T$ and then using
the composition in $\jC(T)$ (where $j_i$ and $j_{i+1}$ are the joints
of $B_i$).  We will see that this map is an isomorphism.  Note that
each of the sets $\jC(\wt{B}_i)(\blank,\blank)$ has an easy
description, as in Lemma \ref{le:C(Delta^n)}; from this one may
extrapolate a corresponding description for $\jC(T)(\blank,\blank)$,
to be explained next.

Let $C_T(a,b)$ denote the poset 
whose elements are subsets of $V_T(a,b)$
which contain $J_T(a,b)$, ordered by inclusion.  
There is a pairing of categories
\[ C_T(b,c)\times C_T(a,b) \ra C_T(a,c) \]
given by union of subsets.

Applying the nerve functor, we obtain a simplicial category $NC_T$
with object set $T_0$.  For $a,b\in T_0$, an $n$-simplex in $NC_T(a,b)$ can be seen as a flag of sets $\vect{T}=T^0\subseteq T^1\subseteq\cdots\subseteq T^n$, where $J_T\subseteq T^0$ and $T^n\subseteq V_T$.

\begin{prop}\label{prop:mapping in necklaces}
Let $T$ be a necklace.  There is a natural isomorphism of simplicial
categories
between $\jC(T)$ and $NC_T$.
\end{prop}

\begin{proof}
Write $T=B_1\Wedge B_2\Wedge \cdots \Wedge B_k$, where the $B_i$'s are
the beads of $T$.
Then 
\[ \jC(T)=\jC(B_1)\amalg_{\jC(*)} \jC(B_2)\amalg_{\jC(*)}\cdots
\amalg_{\jC(*)} \jC(B_k)
\]
since $\jC$ preserves colimits.  Note that $\jC(*)=\jC(\Delta^0)=*$, 
the category with one object and a single morphism (the identity).  

Note that we have isomorphisms $\jC(B_i)\iso NC_{B_i}$ by Lemma
\ref{le:C(Delta^n)}.  We therefore get maps of categories
$\jC(B_i)\ra NC_{B_i} \ra NC_T$, and it is readily checked these
extend to a map $f\colon \jC(T)\ra NC_T$.  To see that this functor is
an isomorphism, it suffices to show that it is fully faithful (as it
is clearly a bijection on objects).

For any $a,b\in T_0$ we will construct an inverse to the map
$f\colon \jC(T)(a,b) \ra NC_T(a,b)$, when $b>a$ (the case $b\leq a$ being obvious).  Let $B_r$ and $B_s$ be the beads containing $a$ and $b$, respectively (if $a$ (resp. $b$) is a joint, let $B_r$ (resp. $B_s$) be the latter (resp. former) of the two beads which contain it).  Let
$j_r,j_{r+1}\ldots,j_{s+1}$ denote the ordered elements of $J_T(a,b)$, indexed so that
$j_i$ and $j_{i+1}$ lie in the bead $B_i$; note that $j_r=a$ and $j_{s+1}=b$.  

Any simplex  $x\in NC_T(a,b)_n$ can be uniquely
written as the composite of $n$-simplices $x_s\circ\cdots\circ
x_{r}$, where $x_i\in NC_T(j_{i},j_{i+1})_n$.  Now $j_{i}$ and
$j_{i+1}$ are vertices within the same bead $B_{i}$ of $T$, therefore
$x_i$ may be regarded as an $n$-simplex in
$\jC(B_i)(j_i,j_{i+1})$.
We then get associated $n$-simplices in $\jC(T)(j_i,j_{i+1})$, and taking
their composite gives an $n$-simplex $\tilde{x}\in\jC(T)(a,b)$.  We
define a map $g\colon NC_T(a,b)\ra \jC(T)(a,b)$ by sending $x$ to
$\tilde{x}$.
One readily checks that this is well-defined
and compatible with the simplicial operators, and it is also clear
that $f\circ g=\id$.  

To see that $f$ is an isomorphism it suffices to now show that $g$ is
surjective.  But upon pondering colimits of categories, it is clear
that every map in $\jC(T)(a,b)$ can be written as a composite of maps
from the $\jC(B_i)$'s.  It follows at once that $g$ is surjective.
\end{proof}

\begin{cor}
\label{cor:mapping space is cube}
Let $T=B_0\Wedge B_1\Wedge \cdots \Wedge B_k$ be a necklace.  Let
$a,b\in T_0$ be such that $a<b$.  Let $j_r,j_{r+1},\ldots,j_{s+1}$ be
the elements of $J_T(a,b)$ (in order), and let $B_i$ denote the bead
containing $j_i$ and $j_{i+1}$, for $r\leq i\leq s$.  Then the map
\[ 
\jC(B_s)(j_s,j_{s+1})\times \cdots \times 
\jC(B_r)(j_r,j_{r+1}) 
\ra \jC(T)(a,b)
\]
is an isomorphism.  Therefore $\jC(T)(a,b)\iso (\Delta^1)^N$ where
$N=|V_T(a,b)-J_T(a,b)|$.  In particular, $\jC(T)(a,b)$ is contractible if $a\leq b$ and empty otherwise.
\end{cor}

\begin{proof}
Follows at once from the previous lemma.
\end{proof}

\begin{remark}\label{rem:heuristic}
Given a necklace $T$, there is a heuristic way to understand faces
(both codimension one and higher) in
the cubes $\jC(T)(a,b)$ in terms of ``paths" from $a$ to $b$ in $T$.
To choose a face in $\jC(T)(a,b)$, one chooses three subsets
$Y,N,M\subset V_T(a,b)$ which cover the set $V_T(a,b)$ and are mutually
disjoint.  The set $Y$ is the set of vertices which we require our
path to go through -- it must contain $J_T(a,b)$; the set $N$ is the
set of vertices which we require our path to not go through; and the
set $M$ is the set of vertices for which we leave the question open.
Such choices determine a unique face in $\jC(T)(a,b)$.  The dimension
of this face is precisely the number of vertices in $M$.
\end{remark}


\section{The categorification functor}\label{sec:jC}

By this point, we fully understand $\jC(\Delta^n)$ as a simplicial category.  Recall that $\jC\colon \sSet \ra s\Cat$ is defined for $S\in\sSet$ by the formula
\[ \jC(S)=\colim_{\Delta^n\ra S} \jC(\Delta^n).\]
The trouble with this formula is that given a diagram $X\colon I\to
s\Cat$ of simplicial categories, it is generally quite difficult to
understand the mapping spaces in the colimit.  In our case, however,
something special happens because the simplicial categories
$\jC(\Delta^n)$ are ``directed" in a certain sense.  It turns out by
making use of necklaces one can write down a precise description of
the mapping spaces for $\jC(S)$; this is the goal of the present
section.

\medskip

Fix a simplicial set $S$ and elements $a,b\in S_0$.  
For any necklace $T$ and map $T\ra S_{a,b}$, there is an induced map
$\jC(T)(\alpha,\omega) \ra \jC(S)(a,b)$.  Let $(\Nec\ovcat S)_{a,b}$
denote the category whose objects are pairs $[T,T\ra S_{a,b}]$ and
whose morphisms are maps of necklaces $T\ra T'$ giving commutative triangles over
$S$.  Then we obtain a  map
\begin{myequation}
\label{eq:colim-map}
 \colim_{T\ra S \in (\Nec\ovcat S)_{a,b}} \Bigl [
\jC(T)(\alpha,\omega)\Bigr ] \lra \jC(S)(a,b).
\end{myequation}
Let us write $E_S(a,b)$ for the domain of this map.  Note that there
are composition maps
\begin{myequation}
\label{eq:composite}
 E_S(b,c) \times E_S(a,b) \lra E_S(a,c) 
\end{myequation}
induced in the following way.  Given $T\ra S_{a,b}$ and $U\ra S_{b,c}$
where $T$ and $U$ are necklaces, one obtains $T\Wedge U\ra S_{a,c}$ in
the evident manner.  The composite
\[\xymatrix{ 
\jC(U)(\alpha_U,\omega_U)\times \jC(T)(\alpha_T,\omega_T) \ar[r] &
\jC(T\Wedge U)(\omega_T,\omega_U)\times \jC(T\Wedge
U)(\alpha_T,\omega_T) \ar[d]^\mu \\
& \jC(T\Wedge U)(\alpha_T,\omega_U)
}
\]
induces the pairing of (\ref{eq:composite}).
One readily checks that $E_S$ is a simplicial category with object set
$S_0$, and (\ref{eq:colim-map}) yields a map of simplicial categories
$E_S \ra \jC(S)$.  Moreover, the construction $E_S$ is clearly functorial in $S$.

Here is our first result:

\begin{prop}
\label{pr:jC-main}
For every simplicial set $S$, the map $E_S\ra \jC(S)$ is an
isomorphism of simplicial categories.
\end{prop}

\begin{proof}
First note that if $S$ is itself a necklace then the identity map
$S\ra S$ is a terminal object in $(\Nec\ovcat S)_{a,b}$.  It follows
at once that $E_S(a,b)\ra \jC(S)(a,b)$ is an isomorphism for all $a$
and $b$.

Now let $S$ be an arbitrary simplicial set, and 
choose vertices $a,b\in S_0$.  We will show
that $E_S(a,b)\to\jC(S)(a,b)$ is a bijection.
Consider the commutative diagram of simplicial sets
$$\xymatrix{\Bigl( \colim_{\Delta^k\to
S}E_{\Delta^k}\Bigr) (a,b)\ar[r]^-t\ar[d]_{\iso}&E_S(a,b)\ar[d]\\
\Bigl(\colim_{\Delta^k\to
S}\jC(\Delta^k)\Bigr)(a,b)\ar@{=}[r]&\jC(S)(a,b).}$$ The bottom equality is the definition of $\jC$.  The left-hand map is an isomorphism by our remarks in the first
paragraph.  It follows that the top map $t$ is injective.  To complete the proof it
therefore suffices to show that $t$ is surjective.

Choose an $n$-simplex $x\in E_S(a,b)_n$; it is represented by a
necklace $T$, a map $f\taking T\to S_{a,b}$, and an element
$\tilde{x}\in \jC(T)(\alpha,\omega)$.
We have a commutative diagram
\[\xymatrix{
\left(
\colim_{\Delta^k\to T}\jC(\Delta^k)\right)(\alpha,\omega)\ar[r] &
\jC(T)(\alpha,\omega) \\
\left(\colim_{\Delta^k\to T}E_{\Delta^k}\right)(\alpha,\omega)\ar[r]\ar[u]
\ar[d]_{f}&E_T(\alpha,\omega)\ar[u]\ar[d]^{E_f}\\
\left(\colim_{\Delta^k\to
S}E_{\Delta^k}\right)(a,b)\ar[r]^-t&E_S(a,b).} 
\]
The $n$-simplex in $E_T(\alpha,\omega)$ represented by
$[T,\id_T\taking T\to T;\tilde{x}]$ is sent to $x$ under $E_f$.  It
suffices to show that the middle horizontal map is surjective, for
then $x$ will be in the image of $t$.  But the top map is
an isomorphism, and the vertical arrows in the top row are
isomorphisms by the remarks from the first paragraph.  Thus, we are done.
\end{proof}

\begin{cor}
\label{co:description1}
For any simplicial set $S$ and elements $a,b\in S_0$, the simplicial
set $\jC(S)(a,b)$ admits the following description. 
An $n$-simplex in $\jC(S)(a,b)$
consists of an equivalence class of triples $[T,T\ra S,\vect{T}]$, where
\begin{itemize}
\item$T$ is a necklace; 
\item $T\ra S$ is a map of
simplicial sets which sends $\alpha_T$ to $a$ and $\omega_T$ to $b$;
and 
\item $\vect{T}$ is a flag of sets $T^0\subseteq
T^1\subseteq\cdots\subseteq T^n$ such that $T^0$ contains the joints
of $T$ and $T^n$ is contained in the set of vertices of $T$.
\end{itemize}
The equivalence
relation is generated by considering $(T\ra S; \vect{T})$ and $(U\ra
S;\vect{U})$ to be equivalent if there exists a map of necklaces $f\colon
T\ra U$ over $S$ with $\vect{U}=f_*(\vect{T})$.

The $i$th face (resp. degeneracy) map omits (resp. repeats) the set
$T^i$ in the flag.
That is, if
$x=(T\ra S;T^0\subseteq\cdots\subseteq T^n)$ represents an $n$-simplex of
$\jC(S)(a,b)$ and $0\leq i\leq n$, then
$$s_i(x)=(T\ra S; T^0\subseteq\cdots\subseteq
T^i\subseteq T^i\subseteq\cdots\subseteq T^n)$$
and
\[
d_i(x)=(T\ra S; T^0\subseteq\cdots\subseteq
T^{i-1}\subseteq T^{i+1}\subseteq\cdots\subseteq T^n).
\]
\end{cor}

\begin{proof}
This is a straightforward interpretation of the colimit  appearing in
the definition of $E_S$ from (\ref{eq:colim-map}).  Recall that  every
colimit can be written as a coequalizer
\[ 
 \colim_{T\ra S \in (\Nec\ovcat S)_{a,b}} \Bigl [
\jC(T)(\alpha,\omega)\Bigr ] \iso \coeq \Bigl[
\coprod_{T_1\ra T_2\ra S} \jC(T_1)(\alpha,\omega)
\dbra \coprod_{T\ra S} \jC(T)(\alpha,\omega)\Bigr],
\]
and that elements of $\jC(T)$ are identified with flags of subsets of
$V_T$, containing $J_T$, by Lemma~\ref{prop:mapping in necklaces}.
\end{proof}

Our next goal is to simplify the equivalence relation appearing in
Corollary~\ref{co:description1} somewhat.  This analysis is somewhat
cumbersome, but culminates in the important Proposition~\ref{pr:Fn-hodiscrete}.

Let us begin by introducing some terminology.  A \dfn{flagged
necklace} is a pair $[T,\vect{T}]$ where $T$ is a necklace and
$\vect{T}$ is a flag of subsets of $V_T$ which all contain $J_T$.  
The \dfn{length of the flag} is the number of subset symbols, or one less
than the number of subsets.  
A morphism of flagged necklaces $[T,\vect{T}]\ra [U,\vect{U}]$ exists
only if the flags have the same length, in which case it is a map of necklaces
$f\colon T\ra U$   such that $f(T^i)=U^i$ for all $i$.
Finally, a flag $\vect{T}=(T^0\subseteq \cdots \subseteq T^n)$ is
called \dfn{flanked} if $T^0=J_T$ and $T^n=V_T$.  Note that if
$[T,\vect{T}]$ and $[U,\vect{U}]$ are both flanked, then every
morphism $[T,\vect{T}]\ra [U,\vect{U}]$ is surjective (because its
image will be a subnecklace of $U$ having the same joints and vertices
as $U$, hence it must be all of $U$).

\begin{lemma}
\label{le:flankify}
Under the equivalence relation of Corollary~\ref{co:description1},
each of the triples $[T,T\ra S,\vect{T}]$ is equivalent to one in which the
flag is flanked.  Moreover, two flanked triples are equivalent (in the
sense of Corollary~\ref{co:description1}) if and only if they can be
connected by a zig-zag of morphisms of flagged necklaces in which
every triple of the zig-zag is flanked.  
\end{lemma}

\begin{proof}
Suppose given a flagged necklace $[T,T^0\subseteq \cdots \subseteq
T^n]$.  There is a unique subnecklace $T'\inc T$ whose set of joints
is $T^0$ and whose vertex set is $T^n$.  Then the pair
$(T',T^0\subseteq \cdots \subseteq T^n)$ is flanked.  This assignment,
which we call {\em flankification}, is actually functorial: a morphism
of flagged necklaces $f\colon[T,\vect{T}]\ra [U,\vect{U}]$ must map
$T'$ into $U'$ and therefore gives a morphism $[T',\vect{T}]\ra
[U',\vect{U}]$.

Using the equivalence relation of Corollary~\ref{co:description1},
each triple $[T,T\ra S,\vect{T}]$ will be equivalent to the flanked
triple $[T',T'\ra T\ra S,\vect{T}]$ via the map $T'\ra T$.  If the
flanked triple $[U,U\ra S,\vect{U}]$ is equivalent to the flanked
triple $[V,V\ra S,\vect{V}]$ then there is a zig-zag of maps between
triples which starts at the first and ends at the second, by Corollary \ref{co:description1}.  Applying
the flankification functor gives a
corresponding zig-zag in which every object is flanked.
\end{proof}

\begin{remark}\label{rem:flanked and outers}
By the previous lemma, we can alter our model for $\jC(S)(a,b)$ so
that the $n$-simplices are equivalence classes of triples $[T,T\ra
S,\vect{T}]$ in which the flag is flanked, and the equivalence
relation is given by maps (which are necessarily surjections) of
flanked triples.  Under this model the degeneracies and inner faces
are given by the same description as before: repeating or omitting one
of the subsets in the flag.  The outer faces $d_0$ and $d_n$ are now
more complicated, however, because omitting the first or last subset
in the flag may produce one which is no longer flanked; one must first
remove the subset and then apply the flankification functor from
Lemma~\ref{le:flankify}.   This model for $\jC(S)(a,b)$ was originally
shown to us by Jacob Lurie; it will play only a very minor role in 
what follows.
\end{remark}

Our next task will be to analyze surjections of flagged triples.  Let
$T$ be a necklace and $S$ a simplicial set.  Say that a map $T\ra S$
is \dfn{totally nondegenerate} if the image of each bead of $T$ is a
nondegenerate simplex of $S$.  Note a totally nondegenerate map need
not be an injection: for example, let $S=\Delta^1/\bd{1}$ and consider
the nondegenerate $1$-simplex $\Delta^1 \ra S$.

Recall that in a simplicial set $S$, if $z\in S$ is a
degenerate simplex then there is a unique nondegenerate simplex $z'$
and a unique degeneracy operator $s_\sigma=s_{i_1}s_{i_2}\cdots
s_{i_k}$ such that $z=s_\sigma(z')$; see \cite[Lemma 15.8.4]{H}.
Using this, and the fact that degeneracy operators correspond to
surjections of simplices, one finds that for any map $T\ra S$ there is
a necklace $\overline{T}$, a map $\overline{T}\ra S$ which is totally
nondegenerate, and a surjection of necklaces $T\ra
\overline{T}$ making the evident triangle commute; moreover, these
three things are unique up to isomorphism.

\begin{prop}
\label{pr:lift}
Let $S$ be a simplicial set and let $a,b\in S_0$.  
\begin{enumerate}[(a)]
\item Suppose that $T$ and $U$ are necklaces, $U\To{u} S$ and $T\To{t}
S$ are two maps, and that $t$ is totally nondegenerate.  Then
there is at most one surjection $f\colon U\fib T$
such that $u=t\circ f$.
\item
Suppose that one
has a diagram
\[ \xymatrix{ U \ar@{->>}[d]_g\ar@{->>}[r]^f & T \ar[d] \\
V\ar[r] & S
}
\]
where $T$, $U$, and $V$ are flagged necklaces, $T\ra S$ is totally
nondegenerate, and $f$ and $g$ are surjections.  Then there
exists a unique map of flagged necklaces $V\ra T$ making the diagram
commute.
\end{enumerate}
\end{prop}

\begin{proof}
We first make the observation that if $A\ra B$ is a surjection of
necklaces and $B\neq *$ then every bead of $B$ is surjected on by a unique bead of
$A$.  Also, each bead of $A$ is either collapsed onto a joint of $B$
or else mapped surjectively onto a bead of $B$.

For (a), note that we may assume $T\neq *$ (or else the claim is
trivial).  Assume there are two distinct surjections $f,f'\colon U\ra T$ such
that $tf=tf'=u$.  Let $B$ be the first bead of $U$ on which $f$ and
$f'$ disagree.  Let $j$ denote the
initial vertex of $U$, and let $C$ be the bead of $T$ whose initial
vertex is $f(j)=f(j')$.

If $u$ maps $B$ to a point in $S$
then $B$ cannot surject onto the bead $C$ (using that $T\ra S$ is
totally nondegenerate); so $B$ must be collapsed to a point by both $f$
and $f'$.  Alternatively, if $u$ does not map $B$ to a point  then $B$ must
surject onto the bead $C$ via both $f$ and $f'$; this
identifies the simplex $B\ra U\ra S$ with a degeneracy of the
nondegenerate simplex $C\ra S$.  Then by uniqueness of degeneracies we
have that $f$ and $f'$ must coincide on $B$, which is a
contradiction.

Next we turn to part (b).  Note that the map $V\ra T$ will necessarily
be surjective, so the uniqueness part is guaranteed by (a); we need
only show existence.  

Observe that if $B$ is a bead in $U$ which maps to a point in $V$
then it maps to a point in $T$, by the reasoning above.  It now follows that
there exists a necklace $U'$, obtained by collapsing every bead of $U$
that maps to a point in $V$, and a commutative diagram
$$\xymatrix@=14pt{U\ar@{->>}[rr]^f\ar@{->>}[rd]\ar@{->>}[dd]_g&&T\ar[dd]\\
&U'\ar@{->>}[ru]^(.4){f'}\ar@{->>}[dl]_(.4){g'}\\V\ar[rr]&&S}$$
Replacing $U, f,$ and $g$ by $U',f',$ and $g'$, and dropping the
primes, we can now assume that $g$ induces a one-to-one correspondence
between beads of $U$ and beads of $V$.  Let $B_1,\ldots,B_m$ denote
the beads of $U$, and let $C_1,\ldots,C_m$ denote the beads of $V$.

Assume that we have constructed the lift $l\colon V\ra T$ on the beads
$C_1,\ldots,C_{i-1}$.  If the bead $B_i$ is mapped by $f$ to a point,
then evidently we can define $l$ to map $C_i$ to this same point and
the diagram will commute.  Otherwise $f$ maps $B_i$ surjectively onto
a certain bead $D$ inside of $T$.  We have the diagram
\[\xymatrix{B_i \ar@{->>}[r]^f \ar@{->>}[d]_{g} 
& D\ar[d]^{t}\\
C_i \ar[r]_{v} & S
}
\] 
where here $f$ and $g$ are surjections between simplices and therefore
represent degeneracy operators $s_f$ and $s_g$.  We have that
$s_f(t)=s_g(v)$.  But the simplex $t$ of $S$ is nondegenerate by
assumption, therefore by \cite[Lemma 15.8.4]{H} we must have $v=s_h(t)$ for
some degeneracy operator $s_h$ such that $s_f=s_gs_h$.  The operator
$s_h$
corresponds to a surjection of simplices $C_i \ra D$ making the  above
square commute, and we define $l$ on $C_i$ to coincide with this map.
Continuing by induction, this produces the desired lift $l$.
It is easy to see that $l$ is a map of flagged necklaces, as
$l(V^i)=l(g(U^i))=f(U^i)=T^i$.  
\end{proof}

\begin{cor}
\label{co:unique-triple}
Let $S$ be a simplicial set and $a,b\in S_0$.  
Under the equivalence relation from Corollary~\ref{co:description1},
every triple $[T,T\ra S_{a,b},\vect{T}]$ is equivalent to a unique triple
$[U,U\ra S_{a,b},\vect{U}]$ which is both flanked and totally nondegenerate.
\end{cor}

\begin{proof}
Let $t=[T,T\ra S_{a,b},\vect{T}]$.  Then $t$ is clearly equivalent to at
least one flanked, totally nondegenerate triple because we can replace
$t$ with $[T',T'\ra S_{a,b},\vect{T}]$ (flankification) and then with
$[\overline{T'},\overline{T'}\ra S_{a,b},\vect{T}']$ (defined above Proposition \ref{pr:lift}).

Now suppose that $[U,U\ra S_{a,b},\vect{U}]$ and $[V,V\ra S_{a,b},\vect{V}]$ are
both flanked, totally nondegenerate, and equivalent in
$\jC(S)(a,b)_n$.  Then by Lemma~\ref{le:flankify} there is a zig-zag
of maps between flanked necklaces (over $S$) connecting $U$ to $V$:
\[ \xymatrixcolsep{1.3pc}
\xymatrix{ &W_1\ar@{->>}[dl]\ar@{->>}[dr] && W_2\ar@{->>}[dl]\ar@{->>}[dr] && \cdots\ar@{->>}[dl]\ar@{->>}[dr] &&
W_k\ar@{->>}[dl]\ar@{->>}[dr] \\
U=U_1 && U_2 && U_3 & \cdots & U_{k}& &\,U_{k+1}=V
}
\]
Using Proposition~\ref{pr:lift}, we inductively construct surjections of
flanked necklaces $U_i\ra U$ over $S$.  This produces a surjection $V\ra U$
over $S$.  Similarly, we obtain a surjection $U\to V$ over $S$.  By
Proposition~\ref{pr:lift}(a) these maps must be inverses of each other; that
is, they are isomorphisms.
\end{proof}

\begin{remark}
Again, as in Remark~\ref{rem:flanked and outers} the above corollary
shows that we can describe $\jC(S)(a,b)$ as the simplicial set whose
$n$-simplices are triples $[T,T\ra S,\vect{T}]$ which are both flanked
and totally nondegenerate.  The degeneracies and inner faces are again
easy to describe---they are repetition or omission of a set in the
flag---but for the outer faces one must first omit a set and then
modify the triple appropriately.  The usefulness of this description
is 
limited because of these complications with the outer faces, but it does
make a brief appearance in
Corollary~\ref{cor:maps in ordered} below.
\end{remark}

The following result is the culmination of our work in this section,
and will turn out to be a key step in the proof of our main theorems.
Fix a simplicial set $S$ and vertices $a,b\in S_0$, and let $F_n$
denote the category of flagged triples over $S_{a,b}$ that have length
$n$.  That is, the objects of $F_n$ are triples $[T,T\ra
S_{a,b},T^0\subseteq \cdots
\subseteq T^n]$ and morphisms are maps of necklaces $f\colon T\ra T'$
over $S$ such that $f(T^i)=(T')^i$ for all $i$.

\begin{prop} 
\label{pr:Fn-hodiscrete}
For each $n\geq 0$, the nerve of $F_n$ is homotopy discrete
in $\sSet_K$.  
\end{prop}

\begin{proof}

Recall from Lemma~\ref{le:flankify} that there is a functor $\phi\colon F_n
\ra F_n$ which sends any triple to its `flankification'.  There is a
natural transformation from $\phi$ to the identity, and the image of
$\phi$ is the subcategory $F_n'\subseteq F_n$ of flanked triples.  It
will therefore suffice to prove that (the nerve of) $F_n'$ is homotopy discrete.  

Recall from Corollary~\ref{co:unique-triple} that every component of
$F_n'$ contains a unique triple $t$ which is both flanked and totally
nondegenerate.  Moreover, following the proof of that corollary
one sees that every triple in the same component as $t$
admits a unique map to $t$---that is to say, $t$ is a final object
for its component.  Therefore its component is contractible.  This
completes the proof.

\end{proof}

\subsection{The functor $\jC$ applied to ordered simplicial sets}~

Note that even if a simplicial set $S$ is small---say, in the sense
that it has finitely many nondegenerate simplices---the space
$\jC(S)(a,b)$ may be quite large.  This is due to the fact that there
are infinitely many necklaces mapping to $S$ (if $S$ is nonempty).  For certain simplicial sets $S$,
however, it is possible to restrict to necklaces which lie {\it inside
of\/} $S$; this cuts down the possibilities.  The following results
and subsequent
example demonstrate this.  Recall the definition 
of ordered simplicial sets from (\ref{de:ordered}).

\begin{lemma}
\label{le:jCL-necklace}
Let $D$ be an ordered simplicial set and let $a,b\in D_0$.  Then every
$n$-simplex in $\jC(D)(a,b)$ is represented by a unique triple
$[T,T\ra D,\vect{T}]$ in which $T$ is a necklace, $\vect{T}$ is a flanked flag of length $n$, and the map $T\ra D$
is injective.
\end{lemma}

\begin{proof}
By Corollary \ref{co:unique-triple}, every $n$-simplex in
$\jC(D)(a,b)$ is represented by a unique triple $[T,T\to D,\vect{T}]$
which is both flanked and totally non-degenerate.  It suffices to show
that if $D$ is ordered, then any totally non-degenerate map $T\to D$
is injective.  This follows from Lemma \ref{lemma:facts on ordered}(6).
\end{proof}

\begin{cor}\label{cor:maps in ordered}

Let $D$ be an ordered simplicial set, and $a,b\in D_0$.  Let
$M_D(a,b)$ denote the simplicial set for which $M_D(a,b)_n$ is the set
of triples $[T,T\To{f} D_{a,b},\vect{T}]$, where $f$ is injective and
$\vect{T}$ is a flanked flag of length $n$; face and boundary maps are
as in Remark \ref{rem:flanked and outers}.  Then there is a natural
isomorphism \[\jC(D)(a,b)\To{\iso} M_D(a,b).\] 
\end{cor}

\begin{proof}

This follows immediately from Lemma \ref{le:jCL-necklace}.

\end{proof}

\begin{ex}
\label{ex:computation}
Consider the simplicial set $S=\Delta^2\amalg_{\Delta^1}\Delta^2$
depicted 
as follows:

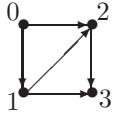
\begin{center}\begin{picture}(40,40)\put(-4,-4){1}\put(31,-3){3}\put(-4,30){0}\put(30,30){2}\put(0,0){$\bullet$}\put(26,0){$\bullet$}\put(0,26){$\bullet$}\put(26,26){$\bullet$}\put(3,2){\vector(1,0){25}}\put(2,28){\vector(0,-1){24}}\put(28,28){\vector(0,-1){24}}\put(2,28){\vector(1,0){25}}\put(3,2){\vector(1,1){25}}\end{picture}\end{center}

We will describe the mapping space $X=\jC(S)(0,3)$ by giving its
non-degenerate simplices and face maps.

By Lemma \ref{le:jCL-necklace}, it suffices to consider flanked
necklaces that inject into $S$.  There are only five such necklaces
that have endpoints $0$ and $3$.  These are
$T=\Delta^1\vee\Delta^1$, which maps to $S$ in two different ways
$f,g$; and $U=\Delta^1\vee\Delta^1\vee\Delta^1$,
$V=\Delta^1\vee\Delta^2$, and $W=\Delta^2\vee\Delta^1$, each of which
maps uniquely into $S_{0,3}$.  The image of $T_0$ under $f$ is
$\{0,1,3\}$ and under $g$ is $\{0,2,3\}$.  The images of $U_0, V_0,$
and $W_0$ are all $\{0,1,2,3\}$.

We find that $X_0$ consists of three elements $[T;\{0,1,3\}]$,
$[T;\{0,2,3\}]$ and $[U;\{0,1,2,3\}]$.  There are two nondegenerate
$1$-simplices, $[V;\{0,1,3\}\subset\{0,1,2,3\}]$ and
$[W;\{0,2,3\}\subset\{0,1,2,3\}]$.  These connect the three
$0$-simplices in the obvious way, resulting in two $1$-simplices with
a common final vertex.  There are no higher non-degenerate simplices.
Thus $\jC(S)(0,3)$ looks like
$$\xymatrix@1{\bullet&\bullet\ar[r]\ar[l]&\bullet}.$$
\end{ex}


\section{Homotopical models for categorification}

In the last section we gave a very explicit description of the mapping
spaces $\jC(S)(a,b)$, for arbitrary simplicial sets $S$ and $a,b\in
S_0$.  While this description was explicit, in some ways it is not
very useful from a homotopical standpoint---in practice it is hard to use this
description to identify the homotopy type of $\jC(S)(a,b)$. 

In this section we will discuss a functor $\jCn\colon\sSet\ra s\Cat$
that has a simpler description than $\jC$ and which is more
homotopical.  We prove that for any simplicial set $S$ there is a
natural zigzag of weak equivalences between $\jC(S)$ and $\jCn(S)$.
Variants of this construction are also introduced, leading to a
collection of functors $\sSet\to\sCat$ all of which are weakly equivalent to $\jC$.
  
\medskip

Let $S\in \sSet$.  A choice of $a,b\in S_0$ will be
regarded as a map $\bd{1}\ra S$.  Let $(\Nec\ovcat S)_{a,b}$ be the
overcategory for the inclusion functor  $\Nec\inc (\bd{1}\ovcat S)$.  
Finally, define
\[ \jCn(S)(a,b)=N(\Nec\ovcat S)_{a,b}.
\]
This is a simplicial category in an evident way.

\begin{remark}
Both the functor $\jC$ and the functor $\jCn$ have distinct advantages
and disadvantages.  The main advantage to $\jC$ is that it is left
adjoint to the coherent nerve functor $N$ (in fact it is a left
Quillen functor $\sSet_J\ra s\Cat$); as such, it preserves
colimits.  However, as mentioned above, the functor $\jC$ can be
difficult to use in practice because the mapping spaces have an
awkward description.

It is at this point that our functor $\jCn$ becomes useful, because
the mapping spaces are given as 
nerves of 1-categories.  Many tools are available for determining
when a morphism between nerves is a Kan equivalence. 
This will be an important point in
\cite{DS}, where we show the $\jC$ functor gives a Quillen equivalence
between $\sSet_J$ and $\sCat$.  See also Section~\ref{se:properties} below.
\end{remark}

Our main theorem is that there is a simple zigzag of weak equivalences
between $\jC(S)$ and $\jCn(S)$; that is, there is a functor
$\jCh\taking\sSet\to\sCat$ and natural weak equivalences $\jC\from
\jCh\to\jCn$.  We begin by describing the functor $\jCh$. 

Fix a simplicial set $S$.   Define $\jCh(S)$ to
have object set $S_0$, and for every $a,b\in S_0$
\[ \jCh(S)(a,b)=\hocolim_{T\in (\Nec\ovcat S)_{a,b}}
\jC(T)(\alpha,\omega).\]
Note the similarities to Theorem~\ref{th:main}, where it was shown
that $\jC(S)(a,b)$ has a similar description in which the hocolim is
replaced by the colim.  In our definition of $\jCh(S)(a,b)$ we mean to
use a particular model for the homotopy colimit, namely the diagonal
of the bisimplicial set whose $(k,l)$-simplices are pairs \[ (F\taking
[k]\to(\Nec\ovcat S)_{a,b}; x\in \jC(F(0))(\alpha,\omega)_l),
\] where $F(0)$ denotes the necklace obtained by 
applying $F$ to $0\in [k]$ and then applying the forgetful functor
$(\Nec\downarrow S)_{a,b}\to\Nec$.  The composition law for $\jCh$ is
defined just as for the $E_S$ construction from Section~\ref{sec:jC}.

We proceed to establish
natural transformations $\jCh\to\jCn$ and $\jCh\to\jC$.
Note that $\jCn(S)(a,b)$ is the homotopy colimit of the constant
functor $\{*\}\taking(\Nec\ovcat S)_{a,b}\to\sSet$ which sends
everything to a point.  
The map $\jCh(S)(a,b) \ra
\jCn(S)(a,b)$ is the map of homotopy colimits induced by the evident
map of diagrams.
 Since the spaces $\jC(T)(\alpha,\omega)$ are all
contractible simplicial sets (see Corollary~\ref{cor:mapping space is
cube}), the induced map $\jCh(S)(a,b)\ra \jCn(S)(a,b)$ is a Kan
equivalence.  We thus obtain a natural weak equivalence of simplicial
categories $\jCh(S)\to\jCn(S)$.

For any diagram in a model category there is a canonical natural
transformation from the homotopy colimit to the colimit of that
diagram.  Hence there is a morphism
\[ \jCh(S)(a,b)\to\colim_{T\in(\Nec\downarrow
S)_{a,b}}\jC(T)(\alpha,\omega) \iso \jC(S)(a,b).
\] 
(For the isomorphism we are using Proposition~\ref{pr:jC-main}.)
As this is natural in $a,b\in S_0$ and
natural in $S$, we have a natural transformation $\jCh\to\jC$.

\begin{thm}\label{th:main}
For every simplicial set $S$, the maps 
$\jC(S)\from \jCh(S)\to\jCn(S)$ defined above are weak equivalences of
simplicial categories.  
\end{thm}

\begin{proof}
We have already established that the natural transformation $\jCh\to\jCn$ is an
objectwise equivalence, so it suffices to show that for each
simplicial set $S$ and objects $a,b\in S_0$ the natural map
$\jCh(S)(a,b)\to\jC(S)(a,b)$ is also a Kan equivalence. 

Recall that $\jCh(S)(a,b)$ is the diagonal of a bisimplicial set whose
$l$th `horizontal' row is the nerve $NF_l$ of the category of flagged
necklaces mapping to $S$, where the flags have length $l$.  
Also recall from Corollary~\ref{co:description1} that $\jC(S)(a,b)$ is the simplicial set which
in level $l$ is $\pi_0(NF_l)$.  But
Proposition~\ref{pr:Fn-hodiscrete} says that $NF_l \ra \pi_0(NF_l)$ is
a Kan equivalence, for every $l$.  It follows that $\jCh(S)(a,b)\ra
\jC(S)(a,b)$ is also a Kan equivalence.
\end{proof}

\subsection{Other models for categorification}
\label{se:gadgets}

One can imagine variations of our basic construction in which one
replaces necklaces with other convenient simplicial sets---which we
might term ``gadgets'', for lack of a better word.  We will see in
Section~\ref{sec:properties of categorification}, for instance, that
using {\it products\/} of necklaces leads to a nice theorem about the
categorification of a product.  In \cite{DS}
several key arguments will hinge on a clever choice of what gadgets to
use.  In the material below we give some basic requirements of the
``gadgets'' which will ensure they give a model equivalent to that of
necklaces.

Suppose $\cP$ is a subcategory of $\bpSet=(\bd{1}\ovcat\sSet)$ containing
the terminal object.  For any simplicial set $S$ and vertices $a,b\in
S_0$, let $(\cP\ovcat S)_{a,b}$ denote the overcategory whose objects
are pairs $[P,P\to S]$, where $P\in\cP$ and the map $P\to S$ sends
$\alpha_P\mapsto a$ and $\omega_P\mapsto b$.  Define
\[\jC^\cP(S)(a,b)=N(\cP\ovcat S)_{a,b}.\] The object $\jC^\cP$ is
simply an assignment which takes a simplicial set $S$ with two
distinguished vertices and produces a ``$\cP$-mapping space."  However, if
$\cP$ is closed under the wedge operation (i.e. for any
$P_1,P_2\in\cP$ one has $P_1\vee P_2\in\cP$), then $\jC^\cP$ may be
given the structure of a functor $\sSet\to\sCat$ in the evident way.

\begin{defn}\label{de:gadgets}

We call a subcategory $\kG\subseteq\bpSet$ a \dfn{category of gadgets} if it satisfies the following properties:
\begin{enumerate}[(1)]
\item $\kG$ contains the category $\Nec$, 
\item For every object $X\in \kG$ and every necklace $T$, all maps
$T\ra X$ are contained in $\kG$, and 
\item For any $X\in \kG$, the simplicial set $\jC(X)(\alpha,\omega)$
is contractible.
\end{enumerate}

The category $\kG$
is said to be \dfn{closed under wedges} if it is also true that
\begin{enumerate}[(4)]
\item For any $X,Y\in \kG$, the wedge $X\Wedge Y$ also belongs to $\kG$.
\end{enumerate}

\end{defn}

The above definition can be generalized somewhat by allowing $\Nec\to\kG$ to be an arbitrary functor over a natural transformation in $\sSet$; we do not need this generality in the present paper.

\begin{prop}
\label{pr:gadget}
Let $\kG$ be a category of gadgets.  Then for any simplicial set $S$
and any $a,b\in S_0$,
the natural map
\[ \jCn(S)(a,b) \lra \jC^{\kG}(S)(a,b) \]
(induced by the inclusion $\Nec\inc \kG$)
is a Kan equivalence.  If $\kG$ is closed under wedges then the map of
simplicial categories $\jCn(S) \ra \jC^{\kG}(S)$ is a weak equivalence.
\end{prop}

\begin{proof}
Let $j\colon (\Nec\ovcat S)_{a,b} \ra (\kG \ovcat S)_{a,b}$ be the
functor induced by the inclusion map $\Nec \inc \kG$.  The map in the
statement of the proposition is just the nerve of $j$.  To verify that
it is a Kan equivalence, it is enough by Quillen's Theorem A \cite{Q}
to verify that all the overcategories of $j$ are contractible.  So fix
an object $[X,X\ra S]$ in $(\kG\ovcat S)_{a,b}$.  The overcategory
$(j\ovcat [X,X\ra S])$ is precisely the category $(\Nec\ovcat
X)_{\alpha,\omega}$, the nerve of which is $\jCn(X)(\alpha,\omega)$.
By Theorem~\ref{th:main} and our assumptions on $\kG$, this is
contractible.

The second statement of the result is a direct consequence of the first.
\end{proof}


\section{Properties of categorification}\label{sec:properties of
categorification}
\label{se:properties}

In this section we establish two main properties of the
categorification functor $\jC$.  First, we prove that there is a
natural weak equivalence $\jC(X\times Y) \he \jC(X)\times \jC(Y)$.
Second, we prove that whenever $S\ra S'$ is a Joyal equivalence
it follows that $\jC(S)\ra \jC(S')$ is a weak equivalence in
$\sCat$.  These properties are also proven in \cite{L}, but the proofs
we give here are of a different nature and make central use of the
$\jCn$ functor.

\medskip

If $T_1,\ldots,T_n$ are necklaces then they are, in particular,
ordered simplicial sets in the sense of Definition~\ref{de:ordered}.
So $T_1\times \cdots \times T_n$ is also ordered, by
Lemma~\ref{lemma:facts on ordered}.  Let $\kG$ be the full subcategory
of $\bpSet=(\bd{1}\ovcat \sSet)$ whose objects are products of necklaces
with a map $f\taking\bd{1}\ra T_1\times\cdots \times T_n$ that has
$f(0)\poleq f(1)$.  

\begin{prop}
\label{pr:products=gadgets}
The category $\kG$ is a category of gadgets in the sense of
Definition~\ref{de:gadgets}.
\end{prop}

For the proof of this one needs to verify that $\jC(T_1\times\cdots
\times T_n)(\alpha,\omega)\he *$.  This is not difficult, but is a bit
of a distraction; we prove it later as Proposition \ref{prop:product of necklaces}.  

\begin{prop}
\label{pr:product}
For any simplicial sets $X$ and $Y$, both $\jC(X\times Y)$ and
$\jC(X)\times \jC(Y)$ are simplicial categories with object set
$X_0\times Y_0$.  For any $a_0,b_0\in X$ and $a_1,b_1\in Y$, the
natural map
\[ \jC(X\times Y)(a_0a_1,b_0b_1) \ra \jC(X)(a_0,b_0)\times
\jC(Y)(a_1,b_1)
\]
induced by $\jC(X\times Y)\ra \jC(X)$ and $\jC(X\times Y)\ra \jC(Y)$
is a Kan equivalence.
Consequently, the map of simplicial categories
\[ \jC(X\times Y) \ra \jC(X)\times \jC(Y)\]
is a weak equivalence in $\sCat$.
\end{prop}

\begin{proof}

Let $\kG$ denote the above category of gadgets, in which the objects
are products of necklaces.  By Theorem~\ref{th:main} and
Proposition~\ref{pr:gadget} 
it suffices to prove the result for $\jC^\kG$ in place of $\jC$.

Consider the functors 
\[\Adjoint{\phi}{(\kG\ovcat X\cross
Y)_{a_0a_1,b_0b_1}}{(\kG\ovcat X)_{a_0,b_0}\cross(\kG\ovcat
Y)_{a_1,b_1}}{\theta}
\]
given by
\[ \phi\colon [G,G\ra X\times
Y] \mapsto \bigl ([G,G\ra X\times Y \ra X],[G,G\ra X\times Y \ra
Y] \bigr )
\]
\and
\[ \theta\colon \bigl ([G,G\ra X],[H,H\ra Y] \bigr ) \mapsto [G\times H,G\times
H\ra X\times Y].
\]
Note that we are using that the subcategory $\kG$ is closed under
finite products.

It is very easy to see that there is a natural transformation $\id \ra
\theta\phi$, obtained by using diagonal maps, and a natural transformation $\phi\theta\ra \id$, obtained by using projections.  As a consequence, the maps $\theta$
and $\phi$ induce inverse homotopy equivalences on the nerves.  This
completes the proof.

\end{proof}

Let $E\taking\Set\to\sSet$ denote the 0-coskeleton functor (see
\cite{AM}).  For any
simplicial set $X$ and set $S$, we have $\sSet(X,ES)=\Set(X_0,S)$.  In
particular, if $n\in\N$ we denote $E^n=E\{0,1,\ldots,n\}$.

\begin{lemma}
\label{le:C(En)}
For any $n\geq 0$, the simplicial category $\jC(E^n)$ is contractible
in $\sCat$---that is to say, all the mapping spaces in $\jC(E^n)$ are
contractible.  
\end{lemma}

\begin{proof}
By Theorem~\ref{th:main} it is sufficient to prove that the mapping
space $\jCn(E^n)(i,j)$ is contractible, for every $i,j\in
\{0,1,\ldots,n\}$.  This  mapping space is the nerve of the
overcategory $(\Nec\ovcat E^n)_{i,j}$.  

Observe that if $T$ is a necklace then any map $T\ra E^n$ extends
uniquely over $\Delta[T]$.  This is because maps into $E^n$ are
determined by what they do on the $0$-skeleton, and $T\inc \Delta[T]$ is an
isomorphism on $0$-skeleta.  

Consider two functors
\[ f,g\colon (\Nec\ovcat E^n)_{i,j} \ra (\Nec\ovcat E^n)_{i,j} \]
given by
\[ f\colon [T,T\llra{x} E^n] \mapsto
[\Delta[T],\Delta[T]\llra{\bar{x}} E^n]
\ \text{and}\ 
g\colon [T,T\llra{x} E^n] \mapsto [\Delta^1,\Delta^1\llra{z} E^n].
\]
Here $\bar{x}$ is the unique extension of $x$ to $\Delta[T]$, and $z$
is the unique $1$-simplex of $E^n$ connecting $i$ to $j$.  Observe
that $g$ is a constant functor.

It is easy to see that there are natural transformations $\id \ra f\la
g$.  The functor $g$  factors through the terminal category $\{*\}$, so after taking nerves the identity map is null homotopic.  
Hence $(\Nec\ovcat E^n)_{i,j}$ is contractible.

\end{proof}

For completeness (and because it is short) we include the following
lemma, established in \cite[Proof of 2.2.5.1]{L}:

\begin{lemma}\label{lemma:jC preserves cofs}

The functor $\jC\taking\sSet\to\sCat$ takes monomorphism to  cofibrations.

\end{lemma}

\begin{proof}
Every cofibration in $\sSet$ is obtained by compositions and cobase
changes from boundary inclusions of simplices.  It therefore suffices
to show that for each $n\geq 0$ the map
$f\taking\jC(\bd{n})\to\jC(\Delta^n)$ is a cofibration in $\sCat$.
Let $0\leq i,j\leq n$.  If $i>0$ or $j<n$ then
$(\Nec\ovcat\Delta^n)_{i,j}\iso(\Nec\ovcat\partial\Delta^n)_{i,j}$,
whereby $$f(i,j)\taking\jC(\bd{n})(i,j)\to\jC(\Delta^n)(i,j)$$ is an
isomorphism by Proposition~\ref{pr:jC-main}.  For the remaining case
$i=0,j=n$, the map $f(0,n)$ is the inclusion of the boundary of a cube
$b\taking\partial((\Delta^1)^{n-1})\to(\Delta^1)^{n-1}$.

Let $U\taking\sSet\to\sCat$ denote the functor which sends a
simplicial set $S$ to the unique simplicial category $U(S)$ with two
objects $x,y$ and morphisms $\Hom(x,x)=\Hom(y,y)=\{*\}$,
$\Hom(y,x)=\emptyset$, and $\Hom(x,y)=S$.  In view of the generating
cofibrations for $\sCat$ (see \cite{Bergner}), it is easy to show that
$U$ preserves cofibrations.  Hence $U(b)$ is a cofibration.  Notice
that $f$ is the pushout of $U(b)$ along the obvious map
$U[\partial((\Delta^1)^{n-1})]\to\jC(\bd{n})$ sending $x\mapsto 0$ and
$y\mapsto n$. Thus, $f$ is a cofibration.

\end{proof}

\begin{lemma}\label{lemma:N preserves fibrants}
If $\cD$ is a simplicial category then $N\cD$ is a quasi-category.
\end{lemma}

\begin{proof}
By adjointness it
suffices to show that each $\jC(j^{n,k})$ is an acyclic cofibration in
$s\Cat$, where $j^{n,k}\colon \Lambda^n_k \inc \Delta^n$ is an inner
horn inclusion ($0<k<n$).  It is a cofibration by Lemma~\ref{lemma:jC
preserves cofs}, so we must only verify that it is a weak equivalence.
Just as in the proof of (\ref{lemma:jC preserves cofs}) above,
$\jC(\Lambda^n_k)(i,j) \ra \jC(\Delta^n)(i,j)$ is an isomorphism
unless $i=0$ and $j=n$.  It only remains to show that
$\jC(\Lambda^n_k)(0,n) \ra
\jC(\Delta^n)(0,n)$ is a Kan equivalence..  An analysis as in
Example~\ref{ex:computation} identifies $\jC(\Lambda^n_k)(0,n)$ with
the result of removing one face from
the boundary of $(\Delta^1)^{n-1}$, which clearly has the same
homotopy type as the cube $(\Delta^1)^{n-1}$.
\end{proof}

\begin{prop}
\label{pr:hoinvariance}
If  $S\ra S'$ is a map of simplicial sets which is a Joyal equivalence then $\jC(S)\ra
\jC(S')$ is a weak equivalence of simplicial categories.
\end{prop}

\begin{proof}

For any simplicial set $X$, the map $\jC(X\times E^n)\ra \jC(X)$
induced by projection is a weak equivalence in $\sCat$.  This follows
by combining Proposition~\ref{pr:product} with Lemma~\ref{le:C(En)}:
\[ \jC(X\times E^n) \llra{\sim} \jC(X)\times \jC(E^n) \llra{\sim}
\jC(X).
\]
Since $X\amalg X \inc X\times E^1$ is a cofibration in $\sSet$, 
 $\jC(X)\amalg \jC(X)=\jC(X\amalg X) \ra \jC(X\times
E^1)$ is a cofibration in $\sCat$, by Lemma \ref{lemma:jC preserves cofs}.  It follows that $\jC(X\times E^1)$
is a cylinder object for $\jC(X)$ in $\sCat$.  So if $\cD$ is a
fibrant simplicial category we may compute
homotopy classes of maps $[\jC(X),\cD]$ as the coequalizer
\[ \coeq\Bigl ( \sCat(\jC(X\times E^1),\cD) \dbra
\sCat(\jC(X),\cD)\Bigr ).
\]
But using the adjunction, this is isomorphic to
\[ \coeq\Bigl ( \sSet(X\times E^1,N\cD) \dbra \sSet(X,N\cD) \Bigr
).
\]
The above coequalizer is
$[X,N\cD]_{E^1}$, and we have identified 
\begin{myequation}
\label{eq:homotopy adjunction} [\jC(X),\cD]\iso [X,N\cD]_{E^1}.
\end{myequation}

Now let $S\ra S'$ be a Joyal equivalence.   Then $\jC(S)\ra
\jC(S')$ is a map between cofibrant objects of $\sCat$.  To prove that
it is a weak equivalence in $\sCat$ it is sufficient to prove that the
induced map on homotopy classes
\[ [\jC(S'),\cD]\ra [\jC(S),\cD]
\]
is a bijection, for every fibrant object $\cD\in \sCat$.  Since $N\cD$
is a quasi-category by Lemma~\ref{lemma:N preserves fibrants} 
and $S\ra S'$ is a Joyal equivalence, we have 
that $[S',N\cD]_{E^1}\to[S,N\cD]_{E^1}$ is a bijection;  the result then follows
by (\ref{eq:homotopy adjunction}).
\end{proof}

\begin{remark}
In fact it turns out that a map of simplicial sets $S\ra S'$ is a
Joyal equivalence {\it if and only if\/} $\jC(S)\ra \jC(S')$ is a weak
equivalence of simplicial categories.  This was proven in \cite{L}, and
will be reproven in \cite{DS} using an extension of the methods from
the present paper.  
\end{remark}



\appendix

\section{Leftover proofs}
\label{se:appendix}
In this section we give two proofs which were postponed in the body
of the paper.  

\subsection{Products of necklaces}

Our first goal is to prove Proposition~\ref{pr:products=gadgets}.  Let
$T_1,\ldots,T_n$ be necklaces, and consider the product
$X=T_1\cross\cdots\cross T_n$.  The main thing we need to prove is
that whenever $a\poleq_X b$ in $X$ the mapping space
$\jC(X)(a,b)\simeq *$ is contractible.

\begin{defn}

An ordered simplicial set $(X,\poleq)$ is called \dfn{strongly
ordered} if, for all $a\poleq b$ in $X$, the mapping space
$\jC(X)(a,b)$ is contractible.

\end{defn}

Note that in any ordered simplicial set $X$ with $a,b\in X_0$, we have
$a\poleq b$ if and only if $\jC(X)(a,b)\neq\emptyset$.  Thus if $X$ is
strongly ordered then its structure as a simplicial category, up
to weak equivalence, is completely determined by the ordering on
its vertices.  We also point out that every necklace $T\in\Nec$ is
strongly ordered by Corollary~\ref{cor:mapping space is cube}.

\begin{lemma}\label{lemma:strongly ordered}

Suppose given a diagram $$X\From{f} A\To{g} Y$$ where $X,Y,$ and $A$
are strongly ordered simplicial sets and both $f$ and $g$ are simple
inclusions. 
Let $B=X\amalg_A Y$ and assume 
the following conditions hold:
\begin{enumerate}[(1)]
\item $A$ has finitely many vertices;
\item  Given any $x\in X$, the set $A_{x\poleq}=\{a\in A \,|\, x\poleq_B a\}$ has
an initial element (an element which is smaller than every other
element).
\item For any $y\in Y$ and $a\in A$, if $y\poleq_Y a$ then $y\in A$.
\end{enumerate}
Then $B$
is strongly ordered.
\end{lemma}

\begin{proof}

By Lemma~\ref{lemma:facts on
ordered}(8), $B$ is an ordered simplicial set and the maps $X\inc B$
and $Y\inc B$ are simple inclusions.  
We must show that for $u,v\in B_0$ with $u\poleq v$, the mapping space
$\jC(B)(u,v)$ is contractible.  Suppose that $u$ and $v$ are both in
$X$; then since $X\inc B$ is simple, any necklace $T\to B_{u,v}$ must
factor through $X$.  It follows that $\jC(B)(u,v)=\jC(X)(u,v)$, which
is contractible since $X$ is strongly ordered.  The case $u,v\in Y$ is
analogous.  We claim we cannot have $u\in Y\backslash A$ and $v\in X\backslash A$.
For if this is so and if $T\ra B$ is a spine connecting $u$ to $v$, then there is a last
vertex $j$ of $T$ that maps into $Y$.  The $1$-simplex leaving that vertex
then cannot belong entirely to $Y$, hence it belongs entirely to $X$.
So $j$ is in both $X$ and $Y$, and hence it is in $A$.  Then we have
$u\poleq j$ and $j\in A$, which by assumption (3) implies $u\in A$, a
contradiction.

The only remaining case to analyze is when $u\in X$ and $v\in Y\backslash A$.
Consider the poset $A_0$ of vertices of $A$, under the relation
$\poleq$.  Let $P$ denote the collection of linearly ordered subsets $S$
of $A_0$ having the property that $u\poleq a\poleq v$ for all $a\in S$.  That
is, each element of $P$ is a chain $u\poleq a_1 \poleq\cdots \poleq
a_n\poleq v$ where each $a_i\in A$.
 We regard $P$ as a
category, where the maps are  inclusions.  
Also let $P_0$ denote the subcategory of $P$ consisting of all
subsets except $\emptyset$. 

Define a functor $D\taking P^{op}\to\Cat$
by sending $S\in P$ to 
\[\{ [T,T\inj B_{u,v}]\hspace{.1in} |
\hspace{.1in} S\subseteq J_T\}, 
\] 
the full subcategory of
$(\Nec\ovcat B)_{u,v}$ spanned by objects $T\To{m} B_{u,v}$ for which
$m$ is an injection and $S\subseteq J_T$.
Let us adopt the notation
\[ M_S(u,v)=\colim_{T\in D(S)} \jC(T)(\alpha,\omega).
\]
Note that there is a natural map
\[ M_\emptyset(u,v) \lra \colim_{T\in (\Nec\ovcat S)_{u,v}}
\jC(T)(\alpha,\omega) \iso \jC(B)(u,v).
\]
The first map is not {\it a priori\/} an isomorphism because in the
definition of $D(\emptyset)$ we require that the map $T\ra B$ be an
injection.  However, using Lemma~\ref{le:jCL-necklace} (or
Corollary~\ref{cor:maps in ordered}) it follows at once that the map
actually is an isomorphism.    

We claim that for each $S$ in $P_0$
the ``latching" map 
$$
L_S\colon \colim_{S'\supset S} M_{S'}(u,v) \ra M_S(u,v)
$$ 
is an injection, where the colimit is over sets $S'\in P$ which strictly
contain $S$.  
To see this, suppose that one has
a triple $[T,T\inc B_{u,v},t\in \jC(T)(\alpha,\omega)_n]$ giving an
$n$-simplex of $M_{S'}(u,v)$ and another triple
$[T',T'\inc B_{u,v},t'\in \jC(U)(\alpha,\omega)_n]$ giving an $n$-simplex
of $M_{S''}(u,v)$.  If these become
identical in $M_S(u,v)$ then it must be that they have the same flankification $\bar{T}=\bar{U}$ and $t=t'$.  Note that every joint of $T$ is a joint of $\bar{T}$, so the joints of $\bar{T}$ include both $S'$ and $S''$.  Because the joints of
any necklace are linearly ordered, it follows that $S'\cup S''$ is
linearly ordered.  Since $T\to\bar{T}$ is an injection, we may consider the triple $[\bar{T},\bar{T}\inc
B_{u,v},t]$ as an $n$-simplex in $M_{S'\cup S''}(u,v)$, which maps to
the two original triples in the colimit;  this proves injectivity.

We claim that the latching map $L_\emptyset\colon \colim_{S\in
P_0^{op}} M_S(u,v)\ra M_\emptyset(u,v)$ is an isomorphism.
Injectivity was established above.  For surjectivity, one needs to
prove that if
$T$ is a necklace and $T\inc B_{u,v}$ is an inclusion, then $T$ must
contain at least one vertex of $A$ as a joint.  To see this, recall
that every simplex of $B$ either lies entirely in $X$ or entirely in
$Y$.  Since $v\notin X$, there is a last joint $j_1$ of $T$ which maps
into $X$.  If $C$ denotes the bead whose initial vertex is $j_1$, then
the image of $C$ can not lie entirely in $X$; so it lies entirely in
$Y$, which means that $j_1$ belongs to both $X$ and $Y$---hence it
belongs to $A$.

From here the argument proceeds as follows.  We will show:
\begin{enumerate}[(i)]
\item The natural map $\hocolim_{S\in P_0^{op}} M_S(u,v)\ra
\colim_{S\in P_0^{op}} M_S(u,v)$ is a Kan equivalence;
\item Each $M_S(u,v)$ is contractible, hence the above homotopy
colimit is Kan equivalent to the nerve of $P_0^{op}$;
\item The nerve of $P_0$ (and hence also $P_0^{op}$) is contractible.
\end{enumerate}
This will prove that $M_\emptyset(u,v)=\jC(B)(u,v)$ is contractible, as desired. 

For (i) we refer to \cite[Section 13]{D} and use the fact that
$P_0^{op}$ has the structure of a directed Reedy category.  Indeed, we
can assign a degree function to $P$ that sends a set $S\subseteq A_0$
to the nonnegative integer $|A_0-S|$; all non-identity morphisms in
$P_0^{op}$ strictly increase this degree.  By \cite[Proposition
13.3]{D} (but with $\Top$ replaced by $\sSet$) it is enough to show that
all the latching maps $L_S$ are cofibrations, and this has already
been established above.

For claim (iii), write $\theta$ for the initial vertex of $A_{u\poleq}$.
Define a functor $F\colon P_0 \ra P_0$ by $F(S)=S\cup \{\theta\}$;
note that $S\cup\{\theta\}$ will be linearly ordered, so this makes
sense.  Clearly there is a natural transformation from the identity
functor to $F$, and also from the constant $\{\theta\}$ functor to $F$.
It readily follows that the identity map on $NP_0$ is homotopic to a
constant map, hence $NP_0$ is contractible.

Finally, for (ii) fix some $S\in P_0$ and let $u=a_0\prec a_1\prec\ldots\prec
a_n\prec a_{n+1}=v$ denote the complete set of elements of
$S\cup\{u,v\}$.
 A necklace $T\inj
B_{u,v}$ whose joints include the elements of $S$ can be split along
the joints, and thus uniquely written as the wedge of necklaces
$T_i\inj B_{a_i,a_{i+1}}$, one for each $0\leq i\leq n$.  Under this
identification, one has
$$
\jC(T)(\alpha,\omega)\iso \jC(T_0)(\alpha_0,\omega_0)
\cross\cdots\cross\jC(T_n)(\alpha_n,\omega_n).
$$ 
Thus $D(p)$ is isomorphic to the category $$
(\Nec\ovcat^m X)_{u,a_1}\cross(\Nec\ovcat^m
A)_{a_1,a_2}\cross\cdots\cross(\Nec\ovcat^m
A)_{a_{n-1},a_n}\cross(\Nec\ovcat^m Y)_{a_n,v}, $$ where
$(\Nec\ovcat^m X)_{s,t}$ denotes the category whose objects are $[T,T\ra X_{s,t}]$
where the map $T\ra X$ is a monomorphism.

Now, it is a general fact
about colimits taken in the category of (simplicial) sets,
that if $M_i$ is a category and
$F_i\taking M_i\to\sSet$ is a functor, for each $i\in \{1,\ldots,n\}$,
then there is an isomorphism of simplicial sets
\begin{align}\tag{A.2.2}\label{dia:colimits over products}
\colim_{M_1\cross\cdots\cross M_n} (F_1\cross\cdots\cross F_n)
\To{\iso}\left(\colim_{M_1}F_1\right)\cross\cdots\cross\left(\colim_{M_n}F_n\right).\end{align}
Applying this in our case, we find that
\[ M_S(u,v) \iso \jC(X)(u,a_1)\times \jC(A)(a_1,a_2)\times \cdots
\times \jC(A)(a_{n-1},a_{n})\times \jC(Y)(a_{n},v).
\]
Note that this is always contractible, since $X, A$,
and $Y$ are strongly ordered.
This proves (ii) and completes the argument.
\end{proof}

\begin{prop}\label{prop:product of necklaces}

Let $T_1,\ldots,T_m$ be necklaces.  Then their product
$T_1\cross\cdots\cross T_m$ is a strongly ordered simplicial set.

\end{prop}

\begin{proof}

We begin with the case $P=\Delta^{n_1}\cross\cdots\cross\Delta^{n_m}$,
where each necklace is a simplex, and show that $P$ is strongly
ordered.  It is ordered by Lemma~\ref{lemma:facts on ordered}, so
choose vertices $a,b\in P_0$ with $a\poleq b$.  If $T$ is a necklace,
any map $T\ra \Delta^j$ extends uniquely to a map
$\Delta[T]\ra \Delta^j$.  It follows that any map $T\ra P_{a,b}$
extends uniquely to $\Delta[T]\ra P_{a,b}$.  Consider the two
functors
\[ f,g\colon (\Nec\ovcat P)_{a,b} \ra (\Nec\ovcat P)_{a,b} \]
where $f$ sends $[T,T\ra P]$ to $[\Delta[T],\Delta[T]\ra P]$ and
$g$ is the constant functor sending everything to
$[\Delta^1,x\colon\Delta^1 \ra P]$ where $x$ is the unique edge of $P$
connecting $a$ and $b$.  Then clearly there are natural
transformations $\id \ra f$ and $g\ra f$, showing that the three maps
$\id$, $f$, and $g$ induce homotopic maps on the nerves.  So the
identity induces the null map, hence 
$\jCn(P)(a,b)=N((\Nec\ovcat P)_{a,b})$ is contractible.  The result for
$P$ now follows by  Theorem~\ref{th:main}.

For the general case, assume by induction that we know the
result for all products of necklaces in which at most $k-1$ of them are
not equal to beads.  The case $k=1$ was handled by the previous
paragraph.
Consider a product
\[ Y=T_1\times \cdots \times T_k \times D\]
where each $T_i$ is a necklace and $D$ is a product of beads.
Write $T_k=B_1\Wedge B_2\Wedge\cdots\Wedge B_r$ where each $B_i$ is a
bead, and let
\[ P_{j}=(T_1\cross\cdots\cross
T_{k-1})\cross(B_{1}\vee\cdots\vee B_{j})\cross D.
\]
We know by induction that $P_1$ is strongly ordered, and we will prove
by a second induction that the same is true for each $P_j$.  So assume
that $P_j$ is strongly ordered for some $1\leq j<r$.

Let us denote $A=(T_1\cross\cdots\cross
T_{k-1})\cross\Delta^0\cross D$
and 
\[ Q=(T_1\cross\cdots\cross
T_{k-1})\cross B_{j+1}\cross D.
\]
Then we have $P_{j+1}=P_{j}\amalg_A Q$, and we know that $P_{j}, A$,
and $Q$ are strongly ordered.  Note that the maps $A\to P_{j+1}$ and
$A\to Q$ are simple inclusions: they are the products of $\Delta^0\to
B_{j}$ (resp. $\Delta^0\to B_{j+1}$) with identity maps, and
any inclusion $\Delta^0\to\Delta^m$ is clearly simple.
It is easy to check that hypothesis (1)--(3) of
Lemma~\ref{lemma:strongly ordered} are satisfied, and so this finishes the proof.
\end{proof}

\begin{proof}[Proof of Proposition~\ref{pr:products=gadgets}]
This follows immediately from Proposition~\ref{prop:product of necklaces}.
\end{proof}

\subsection{The category $\jC(\Delta^n)$}

Our final goal is to give the proof of Lemma~\ref{le:C(Delta^n)}.
Recall that this says there is an isomorphism
$$\jC(\Delta^n)(i,j)\to N(P_{i,j})$$
for $n\in \N$ and $0\leq i,j\leq n$, where $P_{i,j}$ is the poset of subsets of $\{i,i+1,\ldots,j\}$ containing $i$ and $j$.

\begin{proof}[Proof of Lemma~\ref{le:C(Delta^n)}]

The result is obvious when $n=0$, so we assume $n>0$.  Let
$X=(FU)_\bullet([n])(i,j)$ and $Y=P_{i,j}$.  For each $\ell\in\N$, we
will provide an isomorphism $X_\ell\iso Y_\ell$, and these will be
compatible with face and degeneracy maps.

One understands $X_0=FU([n])(i,j)$ as the set of free compositions of
sequences of morphisms in $[n]$ which start at $i$ and end at $j$.  By keeping
track of the set of objects involved in this chain, we identify $X_0$
with the set of subsets of $\{i,i+1,\ldots,j\}$ which contain $i$ and
$j$.  This gives an isomorphism $X_0\to Y_0$.

Similarly, for $\ell>0$ one has that $X_\ell$ is the set of free
compositions of sequences of morphisms in $X_{\ell-1}$.  It is readily
seen that $X_\ell$ (even when $\ell=0$) is in one-to-one
correspondence with the set of ways to ``parenthesize" the sequence
$i,\ldots,j$ in such a way that every element is contained in
$(\ell+1)$-many parentheses (and no closed parenthesis directly
follows an open parenthesis).  Given such a parenthesized sequence,
one can rank the parentheses by ``interiority" (so that interior
parentheses have higher rank).  The face and degeneracy maps on $X$
are given by deleting or repeating all the parentheses of a fixed
rank.

Under this description, a vertex in an $\ell$-simplex of $X$ is given
by choosing a rank and then ignoring all parentheses except those of
that rank.  Then by taking only the last elements before a
close-parenthesis, we get a subset of $\{i+1,\ldots,j\}$ containing
$j$; by unioning with $\{i\}$, we get a well-defined element of $Y_0$.
Given two ranks, the subset of $\{i+1,\ldots,j\}$ corresponding to the
higher rank will contain the subset corresponding to the lower rank.
One also sees immediately that an $\ell$-simplex in $X$ is determined
by its set of vertices, and so we can identify $X_\ell$ with the set
of sequences $S_0\subseteq S_1\subseteq\cdots\subseteq
S_\ell\subseteq\{i,i+1,\ldots,j\}$ containing $i$ and $j$.  This is
precisely the set of $\ell$-simplices of $Y$, so we have our
isomorphism.  It is clearly compatible with face and degeneracy maps.
\end{proof}


\bibliographystyle{amsalpha}

\end{document}